%% file: main.tex
\documentclass{article}
\usepackage[english]{common-paper-style}
\makeatletter\let\c@remark\c@theorem\makeatother
\hypersetup{
  pdftitle={Dimension Thresholds for Relative-Error Representations of Finite Metrics in Pseudo-Euclidean Spaces},
  pdfauthor={Nikita Muravya, Aleksandr Bespalov, Anna Vozianova}
}
\newcommand{\figLa}{(a) landmarks and bunches}
\newcommand{\figLb}{(b) target matrix $\mat{A}^*$}
\newcommand{\figLc}{(c) representation}
\newcommand{\figLlow}{low rank}
\newcommand{\figLsparse}{sparse}
\newcommand{\figLtarget}{target}
\newcommand{\figLrankctl}{rank controlled}
\newcommand{\figLsq}{squared intervals}
\newcommand{\figTlow}{lower bounds without a fixed bound on the spread}
\newcommand{\figTup}{metrics with distances in $[1,5]$: order of the worst-case dimension}
\newcommand{\figTnote}{The order is not determined at the boundary values $4/5$ and $12/13$.}
\newcommand{\figTlowtext}{For every fixed $\varepsilon<1$ there is a metric with $\pdim_\varepsilon(d)=\Omega(N^{c'(\varepsilon)})$, where $c'(\varepsilon)>0$. As $\varepsilon\to1$, the spread of the metrics used may grow.}
\newcommand{\figKy}{Normalized value}
\newcommand{\figKa}{landmarks, $1/x$}
\newcommand{\figKb}{correction, $x$}
\newcommand{\figKsum}{sum}
\renewcommand{\shorttitle}{Relative-Error Representations of Finite Metrics in Pseudo-Euclidean Spaces}

\begin{document}

\title{Dimension Thresholds for Relative-Error Representations of Finite Metrics in Pseudo-Euclidean Spaces}
\author{Nikita Muravya\textsuperscript{\dag}\\ITMO University\\\texttt{nrmuravya@itmo.ru}
\And Aleksandr Bespalov\\ITMO University\\\texttt{sonkot25@mail.ru}
\And Anna Vozianova\\ITMO University\\\texttt{vozianova@itmo.ru}}
\date{}
\maketitle
{\renewcommand{\thefootnote}{\dag}\begin{NoHyper}\footnotetext{Corresponding author.}\end{NoHyper}}

\begin{abstract}
A finite metric space $(V,d)$ with $N$ points is represented in the pseudo-Euclidean space $\R^{p,q}$ with relative error $\varepsilon$ if there are points $\vect{x}_u$ with $|\mathcal{Q}(\vect{x}_u-\vect{x}_v)-d(u,v)^2|\le\varepsilon\,d(u,v)^2$ for all pairs, where $\mathcal{Q}$ is the quadratic form of signature $(p,q)$. Every metric has an exact representation with $p+q\le N-1$, and we ask how small $p+q$ can be made. Our main result is an upper bound for arbitrary metrics in terms of local distance ratios: if every point $u$ has $k$ points, itself included, within $\rho$ times its nearest-neighbour distance, then dimension $O((N/k)\log N+k\rho^4\log N)$ suffices for every fixed $\varepsilon>4/5$. The proof combines a landmark approximation of low rank with a correction of small factorization norm. For the class of metrics with distances in $[1,\Phi]$, $\Phi\ge5$, we bound the worst-case dimension in three ranges of $\varepsilon$. It is linear in $N$ for $\varepsilon<4/5$, lies between $\Omega(\sqrt N)$ and $O(\sqrt N\log N)$ for $4/5<\varepsilon<12/13$, and is of order $\log N$ for $\varepsilon>(\Phi^2-1)/(\Phi^2+1)$. The $O(\sqrt N\log N)$ bound follows from the main result. For $\Phi=5$ the last threshold equals $12/13$, so the three ranges cover all $\varepsilon$ except the boundary values $4/5$ and $12/13$. The lower bounds come from the known count of sign patterns over subgraphs of high-girth graphs; for every fixed $\varepsilon<1$ they give a positive power of $N$. Explicit metrics with a linear bound for $\varepsilon<4/5$ are built from Hadamard matrices, as a consequence of a known bound on the approximate rank. For configurations with null squared intervals a relative error does not reduce the dimension at all: the worst-case dimension is $N-1$ for every $\varepsilon<1$. We also study the parameter $C$ that governs the pseudo-Euclidean Johnson--Lindenstrauss lemma of Deng et~al. We show that the exponent of $C$ in the resulting dimension bound cannot be lowered.
\end{abstract}


\section{Introduction}\label{sec:intro}

\paragraph{The problem.} Let $\R^{p,q}$ be the space $\R^{p+q}$ with the quadratic form
\[
\mathcal{Q}(\vect{x})=\sum_{i\le p}x_i^2-\sum_{p<i\le p+q}x_i^2 .
\]
For a finite metric space $(V,d)$ and $0\le\varepsilon<1$, an \emph{$\varepsilon$-representation} of $d$ in $\R^{p,q}$ is a family of points $\vect{x}_u\in\R^{p,q}$, $u\in V$, such that
\begin{equation}\label{eq:rep}
(1-\varepsilon)\,d(u,v)^2\le\mathcal{Q}(\vect{x}_u-\vect{x}_v)\le(1+\varepsilon)\,d(u,v)^2\qquad(u\ne v).
\end{equation}
We write $\pdim_\varepsilon(d)$ for the least $p+q$ for which such a representation exists. The error is relative and is measured on squared distances. It corresponds to distortion $\sqrt{(1+\varepsilon)/(1-\varepsilon)}$ of the distances themselves.

Pseudo-Euclidean spaces are a natural target for arbitrary dissimilarity data: classical multidimensional scaling applied to a symmetric matrix with zero diagonal produces coordinates whose signs are dictated by the signs of the eigenvalues of the centred matrix~\cite{Gower1985}. In particular, every $N$-point metric has an exact representation with $p+q\le N-1$. The question is whether a relative error allows a much smaller dimension, as it does for Euclidean point sets by the Johnson--Lindenstrauss lemma~\cite{JL1984}.

\paragraph{Background.} Deng, Gao, Lu, Luo and Xin~\cite{Deng2025} proved a Johnson--Lindenstrauss lemma for pseudo-Euclidean spaces: an exact representation $(\vect{x}_u)$ can be projected to dimension $O(\eta^{-2}\log N)$ with error $\eta\,\|\vect{x}_u-\vect{x}_v\|_E^2$, where $\|\cdot\|_E$ is the Euclidean norm of all coordinates. Their Theorem~2.3 restates this as a relative error $\eta\,C_{uv}$ on each pair, where $C_{uv}=\|\vect{x}_u-\vect{x}_v\|_E^2/d(u,v)^2$. Rescaling $\eta$ by $C=\max_{u\ne v}C_{uv}$ gives relative error $\varepsilon$ in dimension $O(C^2\varepsilon^{-2}\log N)$ (Lemma~\ref{lem:deng}). The bound is not stated in this form in~\cite{Deng2025}. The authors note that they know of no lower bound complementing their guarantee.

For general configurations in a pseudo-Euclidean space, whose squared intervals may be zero or negative, a relative error does not help. It preserves null squared intervals exactly, and for an explicit configuration of $N$ points every map with relative error $\varepsilon<1$ needs the same dimension $N-1$ as an exact representation (Theorem~\ref{thm:null}). We therefore turn to metrics, whose squared distances are all positive. For them, linear lower bounds for $\varepsilon<4/5$ follow from a standard counting argument, and explicit examples follow from known bounds on the approximate rank (Section~\ref{sec:lower}). The main question of the paper is what happens for large $\varepsilon$. Unlike configurations with null intervals, metrics of bounded spread do admit a smaller dimension for large $\varepsilon$.

For embeddings of general metrics into normed spaces, the dependence of the dimension on the distortion is well understood. A counting argument whose basic idea goes back to Bourgain~\cite{Bourgain1985}, and in which Matou\v{s}ek~\cite{Matousek1996}, \cite[Section~15.3]{MatousekBook} introduced the explicit use of graphs without short cycles and found the thresholds in the distortion, shows that a single fixed normed space containing all $N$-point metrics with distortion less than $3$ has dimension $\Omega(N)$, with distortion less than $5$ dimension $\Omega(\sqrt N)$, and with distortion less than $7$ dimension $\Omega(N^{1/3})$. On the other hand, every $N$-point metric embeds into $\ell_\infty^{d}$ with distortion $3$ and $d=O(\sqrt N\log N)$ \cite[Corollary~15.3.4 and Theorem~15.7.2]{MatousekBook}. The dimension thus drops from linear to roughly $\sqrt N$ at distortion~$3$. These are statements about one space for all metrics. If the normed space may be chosen separately for each metric, the counting argument does not apply, and the known lower bounds are obtained by a different method~\cite{Matousek1996}. Neither statement applies directly to $\R^{p,q}$: this space is not normed, the form $\mathcal{Q}$ has no bounded unit ball, and the squared $\ell_\infty$ distance is not given by a quadratic form.

The problem can also be stated in matrix terms. The least dimension of an $\varepsilon$-representation differs by at most $2$ from the least rank of a symmetric matrix with zero diagonal that approximates $(d(u,v)^2)$ entrywise with relative error $\varepsilon$ (Proposition~\ref{prop:rank}). This relates the question to the approximate rank of sign matrices~\cite{AlonApproxRank,LeeShraibman2009}, where the factorization norm $\gamma_2$ plays the role that the parameter $C$ plays here.

\paragraph{Results.} Let $\mathcal{M}_N(\Phi)$ be the class of metrics on $N$ points with all distances in $[1,\Phi]$, and let $F_{\Phi,\varepsilon}(N)$ be the largest value of $\pdim_\varepsilon(d)$ over $d\in\mathcal{M}_N(\Phi)$.

\begin{theorem}\label{thm:main}
Let $\Phi\ge5$, $N\ge5$, and $\varepsilon_\Phi=(\Phi^2-1)/(\Phi^2+1)$.
\begin{enumerate}
\item If $0<\varepsilon<4/5$, then $(N-1)/32<F_{\Phi,\varepsilon}(N)\le N-1$.
\item If $4/5<\varepsilon<12/13$, then
\[
\frac1{64}\sqrt{\frac N8}<F_{\Phi,\varepsilon}(N)=O\Bigl(\frac{\Phi^2\sqrt N\,\ln N}{\varepsilon-4/5}\Bigr),
\]
where the constant in $O(\cdot)$ is absolute.
\item If $\varepsilon_\Phi<\varepsilon<1$, then
\[
c(\varepsilon)\ln N-2<F_{\Phi,\varepsilon}(N)\le\Bigl\lceil\frac{64\ln N}{(\Phi^2+1)^2(\varepsilon-\varepsilon_\Phi)^2}\Bigr\rceil,
\]
where $c(\varepsilon)=1/\bigl(4\ln(2e/\ln(1/\varepsilon))\bigr)>0$.
\end{enumerate}
\end{theorem}

For $\Phi=5$ we have $\varepsilon_\Phi=12/13$, and the three items cover all $\varepsilon\in(0,1)$ except the values $4/5$ and $12/13$. The worst-case dimension is of order $N$, then of order $\sqrt N$ up to a logarithmic factor, then of order $\log N$ (Figure~\ref{fig:thresholds}). The value $\varepsilon=4/5$ corresponds to distortion $3$ and $\varepsilon=12/13$ to distortion $5$, the spread of the class. For $\Phi>5$ the range $12/13\le\varepsilon\le\varepsilon_\Phi$ is covered only by the upper bound of item~(2), which holds for all $4/5<\varepsilon<1$, and by weaker lower bounds. The upper bound of item~(2) is an asymptotic statement: the absolute constant given by the proof is large, and the bound is below the trivial bound $N-1$ only for very large $N$.

\begin{figure}[t]
\centering
\input{figures/fig-thresholds.tex}
\caption{Bounds on the worst-case dimension as a function of $\varepsilon$ (the axis is not to scale). Top: the class of metrics with distances in $[1,5]$. The lower and upper bounds have the same order in the first and third ranges and differ by a factor $O(\log N)$ in the second. Open circles mark the boundary values $4/5$ and $12/13$, for which the order is not known. Bottom: lower bounds valid for all metrics. The spread of the metrics attaining them grows as $\varepsilon\to1$.}
\label{fig:thresholds}
\end{figure}

The upper bound of item~(2) is a consequence of our main technical result, Theorem~\ref{thm:upper}, which applies to arbitrary metrics. The lower bounds of items~(1) and~(2) are attained on metrics with values in $\{1,\dots,5\}$ and therefore also bound the worst case over all metrics. Theorem~\ref{thm:main} follows from the more general statements below.
\begin{itemize}
\item \emph{An upper bound in terms of local distance ratios} (Theorem~\ref{thm:upper}). Let $d$ be an arbitrary metric on $N\ge4$ points. For a point $u$ let $\nu(u)$ be the distance to its nearest neighbour and $R_k(u)$ the distance to its $k$-th nearest point, with $u$ itself counted first, and let $\rho_k=\max_uR_k(u)/\nu(u)$. Then for all $2\le k\le N$ and $4/5<\varepsilon<1$
\[
\pdim_\varepsilon(d)=O\Bigl(\frac Nk\ln N+\frac{k\,\rho_k^4\ln N}{(\varepsilon-4/5)^2}\Bigr)
\]
with an absolute constant; Theorem~\ref{thm:upper} gives the bound in explicit form. The bound is invariant under scaling. It involves the spread only through the local ratios $R_k(u)/\nu(u)$, although these can be as large as the spread. The parameter $k$ balances the number of landmarks against the size of the correction (Figure~\ref{fig:tradeoff}).
\item \emph{Lower bounds} (Theorem~\ref{thm:girth}, Corollary~\ref{cor:lower}). If $G$ is a graph with $N$ vertices, $m$ edges and girth greater than $\ell$, and $\varepsilon<(\ell^2-1)/(\ell^2+1)$, then some subgraph $H\subseteq G$ has a truncated graph metric with $\pdim_\varepsilon\ge\lfloor m/(8N)\rfloor+1$. Consequently, for every fixed $\varepsilon<1$ there are $N$-point metrics with $\pdim_\varepsilon=\Omega(N^{c'(\varepsilon)})$, $c'(\varepsilon)>0$.
\item \emph{Null intervals} (Theorem~\ref{thm:null}). There is an explicit configuration of $N$ points in $\R^{\lfloor N/2\rfloor,\lfloor N/2\rfloor}$, all of whose squared intervals vanish except on $\lfloor N/2\rfloor$ disjoint pairs, such that every representation with relative error $\varepsilon<1$ needs dimension $N-1$; this dimension suffices for every configuration of $N$ points.
\item \emph{Explicit metrics} (Section~\ref{sec:explicit}). Bipartite metrics built from Hadamard matrices need dimension $\Omega(N)$ for every fixed $\varepsilon<4/5$, and the metric of the Paley graph for every fixed $\varepsilon<3/5$. These are consequences of known bounds on the approximate rank~\cite{AlonApproxRank,Sherstov2011}, once the dimension is reduced to a rank.
\item \emph{The high-error range} (Section~\ref{sec:high}). Above $\varepsilon_\Phi$ a random sign configuration in dimension $O(\log N)$ represents every metric of $\mathcal{M}_N(\Phi)$. Conversely, the uniform metric requires dimension at least $c(\varepsilon)\ln N-2$ for every $\varepsilon<1$ and every signature.
\item \emph{The parameter $C$} (Section~\ref{sec:C}). We define $C_*(d)$ as the least value of $C$ over all exact representations, express it as a semidefinite program, and bound it from below by a ratio of the trace norm to the trace of the centred Gram matrix. For all sufficiently large $N$ there is a metric with values $1$ and $2$ such that $C_*=\Theta(\sqrt N)$ and $\pdim_\varepsilon=\Omega(N)$, hence $\pdim_\varepsilon=\Omega(C_*^2)$, for all $\varepsilon<3/5$ (Theorem~\ref{thm:Cfamily}; the proof needs only $N\ge1000$). Hence the exponent $2$ of $C$ in the bound of Lemma~\ref{lem:deng} cannot be lowered; the precise statement is given after Theorem~\ref{thm:Cfamily}. A factor $\varepsilon^{-2}\log N$ remains between the lower and upper bounds.
\end{itemize}

\paragraph{Methods.} The proofs rely on known tools. The lower bounds (Theorem~\ref{thm:girth}) follow Matou\v{s}ek's scheme~\cite[Proposition~15.3.3]{MatousekBook} and use the same family of truncated graph metrics of subgraphs of a high-girth graph. What changes is the step that counts essentially different placements of the points. In a normed space the images can be translated and rescaled into a ball of bounded radius and rounded to the points of a finite net, and the rounding changes the distances only slightly. This step does not carry over to $\R^{p,q}$: the form $\mathcal{Q}$ is not a norm, bounded squared intervals do not bound the coordinates of the points, and a small change of the coordinates can change a squared interval by a large amount. We therefore bound the number of possible patterns by the number of sign patterns of quadratic polynomials, using Warren's theorem~\cite{Warren1968}, as in~\cite{Alon1985,AlonApproxRank}. A metric may have representations in different spaces $\R^{p,q}$, but for $p+q\le s$ all of them lie, after zero coordinates are appended, in one of the $s+1$ spaces $\R^{p,s-p}$, so the count remains finite. The explicit lower bounds of Section~\ref{sec:explicit} follow from known bounds on the approximate rank~\cite[Theorem~1.2]{AlonApproxRank}, \cite[Prop.~2.2, 2.3]{Sherstov2011}, once the dimension is reduced to the rank of the matrix of squared intervals shifted by a multiple of $\mat{J}$. The upper bound uses nearest landmarks and bunches, as in the distance oracles of Thorup and Zwick~\cite{ThorupZwick2005}. The remaining entries are controlled by a factorization of small norm followed by a random projection, as in the bounds on approximate rank via $\gamma_2$~\cite{LeeShraibman2009,AlonApproxRank}. The logarithmic lower bound follows the method of Alon~\cite{Alon2009}, which uses Hadamard powers of a perturbed identity matrix. Our contribution is the application of these tools to relative entrywise approximation of squared-distance matrices. The main result is Theorem~\ref{thm:upper}: for every metric on $N\ge4$ points and all $2\le k\le N$, $4/5<\varepsilon<1$, it gives the bound $\pdim_\varepsilon(d)=O\bigl((N/k)\ln N+k\rho_k^4\ln N/(\varepsilon-4/5)^2\bigr)$, together with an explicit form of it. Together with the lower bounds and the results of Section~\ref{sec:high}, it leads to the thresholds in $\varepsilon$ of Theorem~\ref{thm:main} (Figure~\ref{fig:thresholds}). We complement this with the statements on the parameter $C_*$.

\paragraph{Organization.} Section~\ref{sec:prelim} contains the reduction to ranks and two standard lemmas. The upper bound is proved in Section~\ref{sec:upper}, and the lower bounds, including the case of null intervals and the explicit metrics, in Section~\ref{sec:lower}. The high-error range is treated in Section~\ref{sec:high}, which also completes the proof of Theorem~\ref{thm:main}. Section~\ref{sec:C} treats the parameter $C_*$, and Section~\ref{sec:discussion} lists open questions.

\section{Preliminaries}\label{sec:prelim}

Throughout, $\mathcal{Q}$ denotes the quadratic form of the pseudo-Euclidean space, and $\pdim_\varepsilon(d)$ the least signature dimension in which inequality~\eqref{eq:rep} can be satisfied. The parameter $\Phi$ bounds the spread of the metric, and $F_{\Phi,\varepsilon}(N)$ describes the worst case on the corresponding class. The full definitions are given in the introduction; below we recall only the matrix facts needed in the proofs.

For a symmetric matrix $\mat{G}$ of order $N$ put
\[
\mathcal{K}(\mat{G})_{uv}=\mat{G}_{uu}+\mat{G}_{vv}-2\mat{G}_{uv}.
\]
The matrix $\mathcal{K}(\mat{G})$ is symmetric with zero diagonal. If $\mat{G}$ is the Gram matrix of points $\vect{x}_u\in\R^{p,q}$ with respect to the bilinear form of $\mathcal{Q}$, then $\mathcal{Q}(\vect{x}_u-\vect{x}_v)=\mathcal{K}(\mat{G})_{uv}$ and $\rank\mat{G}\le p+q$. Conversely, a symmetric matrix $\mat{G}$ with $p$ positive and $q$ negative eigenvalues is the Gram matrix of points in $\R^{p,q}$: if $\mat{G}=\mat{V}\mat{\Lambda}\mat{V}^{\mathsf T}$ is a spectral decomposition, take the rows of $\mat{V}|\mat{\Lambda}|^{1/2}$ and keep the coordinates that correspond to non-zero eigenvalues.

Let $\vect{1}$ be the all-ones vector, $\mat{J}=\vect{1}\vect{1}^{\mathsf T}$ and $\mat{P}=\mat{I}-\frac1N\mat{J}$. A symmetric matrix $\mat{M}$ is \emph{centred} if $\mat{M}\vect{1}=\vect{0}$. The kernel of $\mathcal{K}$ on symmetric matrices consists of the matrices $\vect{1}\vect{w}^{\mathsf T}+\vect{w}\vect{1}^{\mathsf T}$. The only centred matrix of this form is zero. For a symmetric matrix $\mat{A}$ with zero diagonal, $\mathcal{K}(-\tfrac12\mat{A})=\mat{A}$ and $\mathcal{K}(-\tfrac12\mat{P}\mat{A}\mat{P})=\mat{A}$.

For a metric $d$ on $V$ let $\mat{D}=(d(u,v)^2)$ and let $\mathcal{A}_\varepsilon(d)$ be the set of symmetric matrices $\mat{A}$ with zero diagonal and $|\mat{A}_{uv}-\mat{D}_{uv}|\le\varepsilon\mat{D}_{uv}$.

\begin{proposition}\label{prop:rank}
For every metric $d$ on $N$ points and every $0\le\varepsilon<1$,
\[
\pdim_\varepsilon(d)=\min_{\mat{A}\in\mathcal{A}_\varepsilon(d)}\rank(\mat{P}\mat{A}\mat{P}),\qquad
\min_{\mat{A}\in\mathcal{A}_\varepsilon(d)}\rank\mat{A}-2\ \le\ \pdim_\varepsilon(d)\ \le\ \min_{\mat{A}\in\mathcal{A}_\varepsilon(d)}\rank\mat{A}.
\]
In particular, $\pdim_\varepsilon(d)\le N-1$.
\end{proposition}
\begin{proof}
Points $\vect{x}_u$ form an $\varepsilon$-representation exactly when their matrix of squared intervals lies in $\mathcal{A}_\varepsilon(d)$. Fix $\mat{A}\in\mathcal{A}_\varepsilon(d)$. The symmetric matrices $\mat{G}$ with $\mathcal{K}(\mat{G})=\mat{A}$ are $\mat{G}_0+\vect{1}\vect{w}^{\mathsf T}+\vect{w}\vect{1}^{\mathsf T}$ with $\mat{G}_0=-\frac12\mat{P}\mat{A}\mat{P}$. Since $\mat{P}\vect{1}=\vect{0}$ and $\mat{G}_0$ is centred, $\mat{P}(\mat{G}_0+\vect{1}\vect{w}^{\mathsf T}+\vect{w}\vect{1}^{\mathsf T})\mat{P}=\mat{G}_0$, so every such matrix has rank at least $\rank\mat{G}_0$. Hence the least dimension in which $\mat{A}$ is realized exactly equals $\rank(\mat{P}\mat{A}\mat{P})$. Finally, $\rank(\mat{P}\mat{A}\mat{P})\le\rank\mat{A}$, and $\mat{A}=\vect{g}\vect{1}^{\mathsf T}+\vect{1}\vect{g}^{\mathsf T}-2\mat{G}_0$ with $\vect{g}$ the diagonal of $\mat{G}_0$ gives $\rank\mat{A}\le\rank\mat{G}_0+2$.
\end{proof}

The existence of an exact representation obtained from the centred matrix is classical; see~\cite{Gower1985} and \cite[Proposition~2.1]{Deng2025}. The comparison of the ranks follows from the identities for the operator $\mathcal{K}$ and from the fact that centring can lower the rank by at most two. Moreover, $\pdim_\varepsilon(\lambda d)=\pdim_\varepsilon(d)$ for $\lambda>0$.

The definitions extend to any symmetric matrix $\mat{D}$ with zero diagonal, whose entries may have either sign. Put
\[
\mathcal{A}_\varepsilon(\mat{D})=\bigl\{\mat{A}=\mat{A}^{\mathsf T}:\ \mat{A}_{uu}=0,\ |\mat{A}_{uv}-\mat{D}_{uv}|\le\varepsilon|\mat{D}_{uv}|\bigr\},
\]
and let $\pdim_\varepsilon(\mat{D})$ be the least $p+q$ for which there are points $\vect{x}_u\in\R^{p,q}$ whose matrix of squared intervals lies in $\mathcal{A}_\varepsilon(\mat{D})$. For a metric $d$ this is $\pdim_\varepsilon(d)$ with $\mat{D}=(d(u,v)^2)$. The proof of Proposition~\ref{prop:rank} does not use the triangle inequality and holds for $\pdim_\varepsilon(\mat{D})$ as well. If $\mat{D}$ is the matrix of squared intervals of a configuration $\{\vect{y}_u\}\subset\R^{p_0,q_0}$, then $\pdim_\varepsilon(\mat{D})$ is the least dimension $p'+q'$ of a space $\R^{p',q'}$ into which the configuration can be mapped with relative error $\varepsilon$ on all squared intervals, that is, with $|\mathcal{Q}(f(\vect{y}_u)-f(\vect{y}_v))-\mathcal{Q}(\vect{y}_u-\vect{y}_v)|\le\varepsilon|\mathcal{Q}(\vect{y}_u-\vect{y}_v)|$. The signature $(p',q')$ is free.

\begin{lemma}[Warren~\cite{Warren1968}; see {\cite[Theorem~2.2]{AlonApproxRank}}]\label{lem:warren}
Let $P_1,\dots,P_m$ be real polynomials of degree at most $2$ in $n\le m$ variables. The number of vectors $(\sign P_1(\vect{z}),\dots,\sign P_m(\vect{z}))\in\{-1,1\}^m$, over all $\vect{z}\in\R^n$ at which no $P_i$ vanishes, is at most $(8em/n)^n$.
\end{lemma}

\begin{lemma}[Johnson--Lindenstrauss, in the form of~\cite{DasguptaGupta2003}]\label{lem:jl}
Let $W\subset\R^n$ be a set of $M$ points, $0<\eta<1$, and let $k$ be a positive integer with $k\ge4\ln M/(\eta^2/2-\eta^3/3)$. There is a linear map $\varphi:\R^n\to\R^k$ with $(1-\eta)\|\vect{a}-\vect{b}\|^2\le\|\varphi\vect{a}-\varphi\vect{b}\|^2\le(1+\eta)\|\vect{a}-\vect{b}\|^2$ for all $\vect{a},\vect{b}\in W$.
\end{lemma}

This is~\cite[Theorem~2.1]{DasguptaGupta2003}. Since $\eta^2/2-\eta^3/3\ge\eta^2/6$ for $0<\eta<1$, the condition holds whenever $k\ge24\ln M/\eta^2$; we use this simpler form only in Section~\ref{sec:C}. In~\cite{DasguptaGupta2003} the number of points is denoted by $n$ and the dimension of the ambient space by $d$. Here we use the notation of Lemma~\ref{lem:jl}. The statement of Theorem~2.1 does not assert that the map is linear, but the proof gives a linear map. If $n\le k$, one can take the embedding of $\R^n$ into $\R^k$ that pads vectors with zeros. If $n>k$, one takes the orthogonal projection onto a random $k$-dimensional subspace, identified with $\R^k$, and multiplies it by $\sqrt{n/k}$. With positive probability this map satisfies the required inequalities.

\begin{corollary}\label{cor:inner}
Let $\vect{a}_1,\dots,\vect{a}_N\in\R^n$ satisfy $\|\vect{a}_u\|^2\le\gamma$, let $0<\eta<1$, and let $k$ be a positive integer with $k\ge4\ln(N^2+1)/(\eta^2/2-\eta^3/3)$. There is a linear map $\varphi:\R^n\to\R^k$ with $|\langle\varphi\vect{a}_u,\varphi\vect{a}_v\rangle-\langle\vect{a}_u,\vect{a}_v\rangle|\le\eta\gamma$ for all $u,v$, including $u=v$.
\end{corollary}
\begin{proof}
Apply Lemma~\ref{lem:jl} to the $M=N^2+1$ points $\vect{0}$, $\vect{a}_u$, $\vect{a}_u\pm\vect{a}_v$ ($u<v$). The norms of these vectors are preserved up to the factor $1\pm\eta$, and $\langle\vect{a},\vect{b}\rangle=\frac14(\|\vect{a}+\vect{b}\|^2-\|\vect{a}-\vect{b}\|^2)$ gives an error of at most $\frac\eta4(\|\vect{a}+\vect{b}\|^2+\|\vect{a}-\vect{b}\|^2)=\frac\eta2(\|\vect{a}\|^2+\|\vect{b}\|^2)\le\eta\gamma$.
\end{proof}

\section{An upper bound in terms of local distance ratios}\label{sec:upper}

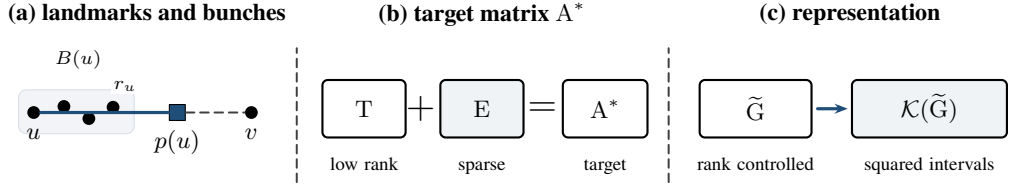
\begin{figure}[t]
\centering
\input{figures/fig-landmarks.tex}
\caption{The construction in the proof of Theorem~\ref{thm:upper}. (a) Landmarks (squares), the nearest landmark $p(u)$ and the bunch $B(u)$ of points closer to $u$ than $p(u)$. For $v\notin B(u)$ the quantity $(r_u+d(p(u),v))^2$ lies between $d(u,v)^2$ and $9d(u,v)^2$. (b) These quantities form a matrix $\mat{T}$ of rank at most $2|L|+1$; the correction $\mat{E}$ is supported on the pairs with $v\in B(u)$ or $u\in B(v)$ and on the diagonal. It is sparse but not of low rank. Its rows have small norm, so it can be compressed by a random projection. (c) The operator $\mathcal{K}$ restores the zero diagonal, and the spectral decomposition of $\widetilde{\mat{G}}$ gives the points.}
\label{fig:landmarks}
\end{figure}

Let $(V,d)$ be a metric space with $N\ge2$ points. For $u\in V$ let
\[
\nu(u)=\min_{v\ne u}d(u,v)
\]
be the distance to the nearest neighbour, and for $1\le k\le N$ let $R_k(u)$ be the least radius of a closed ball centred at $u$ that contains at least $k$ points, the centre included. Equivalently, $R_k(u)$ is the distance from $u$ to its $k$-th nearest point, with $u$ itself counted first. Put
\[
\rho_k=\rho_k(d)=\max_{u\in V}\frac{R_k(u)}{\nu(u)} .
\]
Thus $\rho_1=0$, $\rho_2=1$, $\rho_k\ge1$ for $k\ge2$, and $\rho_k$ is invariant under scaling of the metric. If all distances lie in $[1,\Phi]$, then $\rho_k\le\Phi$ for all $k$. In general $\rho_k$ can be large even for metrics of small dimension: the points $1,2,4,\dots,2^{n-1}$ of the real line have $\rho_k=2^{k-1}-1$ for $2\le k\le n$.

\begin{theorem}\label{thm:upper}
Let $d$ be a metric on $N\ge4$ points, let $2\le k\le N$ and $4/5<\varepsilon<1$. Put
\[
\delta=5\varepsilon-4,\qquad \eta=\frac{\delta}{68\sqrt{k-1}\,\rho_k^2}.
\]
Then $0<\eta<1$ and
\[
\pdim_\varepsilon(d)\ \le\ 4\Bigl\lceil\frac Nk\ln N\Bigr\rceil+2\Bigl\lceil\frac{4\ln(4N^2+1)}{\eta^2/2-\eta^3/3}\Bigr\rceil+2 .
\]
In particular, $\pdim_\varepsilon(d)=O\bigl((N/k)\ln N+k\rho_k^4\ln N/(\varepsilon-4/5)^2\bigr)$ with an absolute constant.
\end{theorem}

The quantity $\delta$ is not a free parameter. The construction below produces a matrix whose entries lie between $(1-\delta)d(u,v)^2$ and $(9+\delta)d(u,v)^2$, where the factor $9$ comes from the triangle inequality, and the required relative error $\varepsilon$ is reached exactly when $(9+\delta)/(1-\delta)=(1+\varepsilon)/(1-\varepsilon)$. This equation gives $\delta=5\varepsilon-4=5(\varepsilon-\frac45)$, which lies in $(0,1)$ exactly for $4/5<\varepsilon<1$.

\begin{corollary}\label{cor:spread}
Let $d$ be a metric on $N\ge4$ points with all distances in $[1,\Phi]$ and let $4/5<\varepsilon<1$. Then
\[
\pdim_\varepsilon(d)=O\Bigl(\frac{\Phi^2\sqrt N\,\ln N}{\varepsilon-4/5}\Bigr)
\]
with an absolute constant.
\end{corollary}
\begin{proof}
Let $K$ be the absolute constant of Theorem~\ref{thm:upper}, so that $\pdim_\varepsilon(d)\le K\bigl((N/k)\ln N+k\rho_k^4\ln N/\delta^2\bigr)$ for $2\le k\le N$, where $\delta=5\varepsilon-4$. Put $X=\Phi^2\sqrt N\ln N/\delta$ and $k_0=\delta\sqrt N/\Phi^2$; the value $k_0$ balances the two terms when $\rho_k$ is replaced by $\Phi$. If $k_0\le1$, then $\sqrt N\le\Phi^2/\delta$, so $N\le\Phi^2\sqrt N/\delta<X$ because $\ln N>1$, and the claim follows from $\pdim_\varepsilon(d)\le N-1$. Otherwise let $k=\lceil k_0\rceil$. Then $2\le k\le2k_0$, and $k\le N$ because $k_0\le\sqrt N$. Since $\rho_k\le\Phi$, we get $(N/k)\ln N\le(N/k_0)\ln N=X$ and $k\rho_k^4\ln N/\delta^2\le2k_0\Phi^4\ln N/\delta^2=2X$, hence $\pdim_\varepsilon(d)\le3KX$. In both cases $\pdim_\varepsilon(d)\le\max(1,3K)\,X$, and $\delta=5(\varepsilon-\frac45)$.
\end{proof}

The choice of $k$ in the proof balances the two terms of Theorem~\ref{thm:upper} (Figure~\ref{fig:tradeoff}). Theorem~\ref{thm:upper} also applies to metrics whose spread is not bounded. For example, let $V$ be a union of clusters with at least $k$ points each, such that inside every cluster the ratio of the largest to the smallest distance is at most $\rho$ and the nearest neighbour of every point lies in its own cluster. The diameters of the clusters and the distances between them are otherwise arbitrary. Then $R_k(u)$ is at most the diameter of the cluster of $u$ and $\nu(u)$ is at least its smallest distance, so $\rho_k\le\rho$.

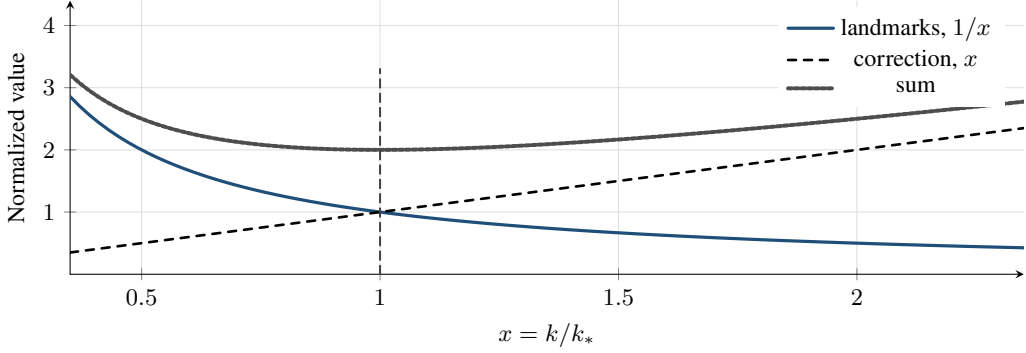
\begin{figure}[t]
\centering
\input{figures/fig-tradeoff.tex}
\caption{The balance behind Corollary~\ref{cor:spread}. Up to absolute constants, the bound of Theorem~\ref{thm:upper} is $A/k+Bk$ with $A=N\ln N$ and $B=\rho^4\ln N/\delta^2$, where $\rho_k\le\rho$ for all $k$. With $k=x\,k_*$, $k_*=\sqrt{A/B}=\delta\sqrt N/\rho^2$, both terms divided by $\sqrt{AB}$ become $1/x$ and $x$. The first term counts landmarks and decreases with $k$; the second measures the correction and increases. Their sum is minimal at $x=1$, with value $2\sqrt{AB}=2\rho^2\sqrt N\ln N/\delta$. Linear axes; the variable $x$ is normalized so that the point $x=1$ corresponds to the optimal balance.}
\label{fig:tradeoff}
\end{figure}

\begin{proof}[Proof of Theorem~\ref{thm:upper}]
Write $\rho=\rho_k$; since $k\ge2$, we have $\rho\ge1$. As $0<\delta<1$ and $k-1\ge1$, also $0<\eta<1$. The construction is shown in Figure~\ref{fig:landmarks}.

\emph{Step 1: landmarks.} For $u\in V$ let $S_u$ be the set of the $k$ points nearest to $u$, the point $u$ included; ties are broken by a fixed order of $V$. There is a set $L\subseteq V$ with $|L|\le\lceil(N/k)\ln N\rceil$ that meets every $S_u$. Indeed, a sample of $\lceil(N/k)\ln N\rceil$ points chosen independently and uniformly misses a given $S_u$ with probability $(1-k/N)^{\lceil(N/k)\ln N\rceil}<1/N$. So with positive probability it meets all of them. Let $r_u=\min_{w\in L}d(u,w)$, let $p(u)\in L$ be a landmark at distance $r_u$, and let
\[
B(u)=\{v\in V:\ d(u,v)<r_u\}.
\]
If $u\in L$, then $r_u=0$ and $B(u)=\varnothing$; otherwise $u\in B(u)$. Since $S_u$ lies in the ball of radius $R_k(u)$ around $u$ and contains a landmark, $r_u\le R_k(u)\le\rho\,\nu(u)$. Every point of $B(u)$ is closer to $u$ than that landmark, hence belongs to $S_u$ and differs from it, so $|B(u)|\le k-1$.

\emph{Step 2: far pairs.} Let $\mat{T}_{uv}=(r_u+d(p(u),v))^2$. By the triangle inequality $\mat{T}_{uv}\ge d(u,v)^2$ for all $u,v$. If $v\notin B(u)$, then $d(p(u),v)\le r_u+d(u,v)\le2d(u,v)$ and $\mat{T}_{uv}\le9\,d(u,v)^2$. Put $g_w(v)=d(w,v)$. Then $\mat{T}_{uv}=r_u^2+2r_u\,g_{p(u)}(v)+g_{p(u)}(v)^2$, so every row of $\mat{T}$ is a linear combination of $\vect{1}^{\mathsf T}$ and the $2|L|$ row vectors $g_w$, $g_w^2$, $w\in L$. Hence $\rank\mat{T}\le2|L|+1$.

\emph{Step 3: the target matrix.} Call a pair $u\ne v$ \emph{far} if $v\notin B(u)$ and $u\notin B(v)$, and \emph{near} otherwise. Define a symmetric matrix $\mat{A}^*$ with zero diagonal by $\mat{A}^*_{uv}=\frac12(\mat{T}_{uv}+\mat{T}_{vu})$ for far pairs and $\mat{A}^*_{uv}=d(u,v)^2$ for near pairs. Then
\begin{equation}\label{eq:Astar}
d(u,v)^2\le\mat{A}^*_{uv}\le9\,d(u,v)^2\qquad(u\ne v).
\end{equation}

\emph{Step 4: the correction.} Let $\mat{E}=\mat{A}^*-\frac12(\mat{T}+\mat{T}^{\mathsf T})$. It vanishes on far pairs. Let $v\in B(u)$, $v\ne u$, so that $t:=d(u,v)<r_u$. Then $\mat{T}_{uv}\le(2r_u+t)^2\le9r_u^2$; also $r_v\le d(v,p(u))\le t+r_u<2r_u$ and $\mat{T}_{vu}\le(2r_v+t)^2\le25r_u^2$. As $0\le\mat{A}^*_{uv}\le\frac12(\mat{T}_{uv}+\mat{T}_{vu})$, we get $|\mat{E}_{uv}|\le17r_u^2$; by symmetry, an entry with $u\in B(v)$ has modulus at most $17r_v^2$. On the diagonal $\mat{E}_{uu}=-\mat{T}_{uu}=-4r_u^2$. Split $\mat{E}=\mat{E}_a+\mat{E}_a^{\mathsf T}$, where row $u$ of $\mat{E}_a$ is supported on $B(u)$: an entry with $v\in B(u)$, $u\notin B(v)$ is assigned to row $u$, and an entry with $v\in B(u)$ and $u\in B(v)$, in particular a diagonal entry, is halved. Let $\vect{y}_u\in\R^N$ be row $u$ of $\mat{E}_a$. Its support lies in $B(u)$, so it has at most $k-1$ non-zero entries. Each of them has modulus at most $17r_u^2$: off the diagonal by the bound above, and on the diagonal because $\frac12|\mat{E}_{uu}|=2r_u^2$. Hence
\begin{equation}\label{eq:rownorm}
\|\vect{y}_u\|\le17\sqrt{k-1}\,r_u^2\le17\sqrt{k-1}\,\rho^2\,\nu(u)^2 .
\end{equation}

\emph{Step 5: projection.} Let $\vect{e}_v$ be the standard basis of $\R^N$, so that $\langle\vect{y}_u,\vect{e}_v\rangle=(\mat{E}_a)_{uv}$. Let $k'=\lceil4\ln(4N^2+1)/(\eta^2/2-\eta^3/3)\rceil$. Apply Corollary~\ref{cor:inner}, with $2N$ in place of $N$ and with $\gamma=1$, to the at most $2N$ unit vectors $\vect{y}_u/\|\vect{y}_u\|$ (for $\vect{y}_u\ne\vect{0}$) and $\vect{e}_v$: there is a linear map $\varphi:\R^N\to\R^{k'}$ with
\[
\bigl|\langle\varphi\vect{y}_u,\varphi\vect{e}_v\rangle-(\mat{E}_a)_{uv}\bigr|\le\eta\,\|\vect{y}_u\|\qquad\text{for all }u,v .
\]
Let $\mat{F}$ be the matrix with entries $\mat{F}_{uv}=\langle\varphi\vect{y}_u,\varphi\vect{e}_v\rangle$; its rank is at most $k'$.

\emph{Step 6: assembly.} Let $\mat{G}_1=-\frac14(\mat{T}+\mat{T}^{\mathsf T})$, $\widetilde{\mat{G}}=\mat{G}_1-\frac12(\mat{F}+\mat{F}^{\mathsf T})$ and $\widetilde{\mat{A}}=\mathcal{K}(\widetilde{\mat{G}})$. The matrix $\widetilde{\mat{A}}$ is symmetric with zero diagonal, and $\rank\widetilde{\mat{G}}\le4|L|+2+2k'$. The matrix $\mat{G}=\mat{G}_1-\frac12(\mat{E}_a+\mat{E}_a^{\mathsf T})$ equals $-\frac12\mat{A}^*$, so $\mathcal{K}(\mat{G})=\mat{A}^*$. The entry $(u,v)$ of $\mat{\Delta}=\widetilde{\mat{G}}-\mat{G}$ has modulus at most $\frac12\eta(\|\vect{y}_u\|+\|\vect{y}_v\|)$, in particular $|\mat{\Delta}_{uu}|\le\eta\|\vect{y}_u\|$. Therefore, for $u\ne v$,
\[
|\widetilde{\mat{A}}_{uv}-\mat{A}^*_{uv}|\le|\mat{\Delta}_{uu}|+|\mat{\Delta}_{vv}|+2|\mat{\Delta}_{uv}|\le2\eta\bigl(\|\vect{y}_u\|+\|\vect{y}_v\|\bigr).
\]
By~\eqref{eq:rownorm} and $\nu(u),\nu(v)\le d(u,v)$ the right-hand side is at most
\[
2\eta\cdot2\cdot17\sqrt{k-1}\,\rho^2\,d(u,v)^2=68\,\eta\sqrt{k-1}\,\rho^2\,d(u,v)^2=\delta\,d(u,v)^2 .
\]
With~\eqref{eq:Astar} this gives $(1-\delta)\,d(u,v)^2\le\widetilde{\mat{A}}_{uv}\le(9+\delta)\,d(u,v)^2$.

\emph{Step 7: normalization.} By the choice of $\delta$ we have $(9+\delta)/(1-\delta)=(1+\varepsilon)/(1-\varepsilon)$. Hence $\sigma\widetilde{\mat{A}}$ with $\sigma=(1-\varepsilon)/(1-\delta)$ lies in $\mathcal{A}_\varepsilon(d)$. Since $\sigma\widetilde{\mat{A}}=\mathcal{K}(\sigma\widetilde{\mat{G}})$, it is the matrix of squared intervals of points in $\R^{p,q}$ with $p+q=\rank\widetilde{\mat{G}}$.

\emph{Step 8: count.} We have $\pdim_\varepsilon(d)\le\rank\widetilde{\mat{G}}\le4|L|+2+2k'$, which is the stated bound. For the asymptotic form, $\eta^2/2-\eta^3/3\ge\eta^2/6$ and $4N^2+1\le N^4$ give $k'\le24\ln(N^4)/\eta^2+1=96\ln N/\eta^2+1$, while $1/\eta^2=68^2(k-1)\rho^4/\delta^2$ and $\delta^2=25(\varepsilon-\frac45)^2$.
\end{proof}

\begin{remark}
The landmarks can also be chosen greedily: at each step, add a point that lies in the largest number of sets $S_u$ not yet met. Each $S_u$ has $k$ elements, so if $U$ sets are not yet met, some point lies in at least $kU/N$ of them. After $t$ steps at most $N(1-k/N)^t$ sets remain, and this number is less than $1$ for $t=\lceil(N/k)\ln N\rceil$. Hence the greedy choice gives the same bound $|L|\le\lceil(N/k)\ln N\rceil$ as Step~1. If the projection is chosen at random, the construction becomes a randomized polynomial-time algorithm. All steps except the projection are exact. For every metric, the symmetric matrix $\frac12(\mat{T}+\mat{T}^{\mathsf T})$ of the proof has rank at most $4|L|+2$ and lies between $d^2$ and $9d^2$ on all pairs except at most $N(k-1)$ of them.
\end{remark}

\begin{remark}\label{rem:barrier}
The ratio $\rho_k$ enters through the size of the correction: by~\eqref{eq:rownorm} the row of $u$ has norm of order $\sqrt k\,r_u^2$, the projection produces an error proportional to this norm, and the admissible error on a pair is of order $d(u,v)^2\ge\nu(u)^2$. If the bunch $B(u)$ contains points at distances much smaller than $r_u$, the argument loses the factor $\rho_k^4$. We do not know whether this factor is necessary.
\end{remark}

\section{Lower bounds}\label{sec:lower}

We first treat configurations with null intervals, where the answer is immediate. Then we give the counting lower bounds (Section~\ref{sec:girth}) and the explicit metrics (Section~\ref{sec:explicit}).

\subsection{Null intervals}\label{sec:null}

If $\mat{D}_{uv}=0$, the condition $|\mat{A}_{uv}-\mat{D}_{uv}|\le\varepsilon|\mat{D}_{uv}|$ forces $\mat{A}_{uv}=0$, and if $\mat{D}_{uv}\ne0$ and $\varepsilon<1$, then $\mat{A}_{uv}$ has the sign of $\mat{D}_{uv}$. So a relative error preserves null squared intervals exactly. The next theorem shows that in the presence of null intervals the error may not reduce the dimension at all.

\begin{theorem}\label{thm:null}
Let $N\ge2$, $k=\lfloor N/2\rfloor$, and let $\vect{e}_1,\dots,\vect{e}_k$ be the standard basis of $\R^k$. In $\R^{k,k}=\R^k\oplus\R^k$ consider the points
\[
\vect{u}_i=\tfrac12(\vect{e}_i,\vect{e}_i),\qquad \vect{v}_i=\tfrac12(-\vect{e}_i,\vect{e}_i),\qquad i=1,\dots,k,
\]
and for odd $N$ add the origin $\vect{o}$. Let $\mat{D}_N$ be the matrix of squared intervals of these $N$ points. Then for every $0\le\varepsilon<1$
\[
\pdim_\varepsilon(\mat{D}_N)=N-1=\max_{\mat{D}}\pdim_\varepsilon(\mat{D}),
\]
where the maximum is over all symmetric matrices $\mat{D}$ of order $N$ with zero diagonal.
\end{theorem}
\begin{proof}
By Proposition~\ref{prop:rank}, every $\mat{D}$ satisfies $\pdim_\varepsilon(\mat{D})\le\rank(\mat{P}\mat{D}\mat{P})\le N-1$, since $\mat{D}\in\mathcal{A}_\varepsilon(\mat{D})$ and $\mat{P}\vect{1}=\vect{0}$. It remains to show that $\pdim_\varepsilon(\mat{D}_N)\ge N-1$.

For $i\ne j$ the differences $\vect{u}_i-\vect{u}_j=\frac12(\vect{e}_i-\vect{e}_j,\vect{e}_i-\vect{e}_j)$, $\vect{v}_i-\vect{v}_j=\frac12(\vect{e}_j-\vect{e}_i,\vect{e}_i-\vect{e}_j)$ and $\vect{u}_i-\vect{v}_j=\frac12(\vect{e}_i+\vect{e}_j,\vect{e}_i-\vect{e}_j)$ have $\mathcal{Q}=\frac14(2-2)=0$. Moreover, $\mathcal{Q}(\vect{u}_i-\vect{v}_i)=\mathcal{Q}(\vect{e}_i,\vect{0})=1$ and $\mathcal{Q}(\vect{u}_i)=\mathcal{Q}(\vect{v}_i)=\frac14-\frac14=0$. Thus all squared intervals between distinct points vanish, except $\mathcal{Q}(\vect{u}_i-\vect{v}_i)=1$.

Let points $\vect{y}_a\in\R^{p,q}$, indexed by the points $a$ of the configuration, have a matrix of squared intervals $\mat{A}\in\mathcal{A}_\varepsilon(\mat{D}_N)$. Then $\mat{A}_{ab}=0$ whenever $(\mat{D}_N)_{ab}=0$, and $\lambda_i=\mat{A}_{\vect{u}_i\vect{v}_i}\in[1-\varepsilon,1+\varepsilon]$, so $\lambda_i>0$. We show that the points $\vect{y}_a$ are affinely independent, which gives $p+q\ge N-1$. Let $\sum_ac_a=0$ and $\sum_ac_a\vect{y}_a=\vect{0}$. Let $B$ be the symmetric bilinear form with $B(\vect{x},\vect{x})=\mathcal{Q}(\vect{x})$, and put $S=\sum_bc_b\mathcal{Q}(\vect{y}_b)$. For every $a$,
\[
\sum_bc_b\mat{A}_{ab}=\sum_bc_b\bigl(\mathcal{Q}(\vect{y}_a)+\mathcal{Q}(\vect{y}_b)-2B(\vect{y}_a,\vect{y}_b)\bigr)=S .
\]
The only nonzero entry in the row $\vect{u}_i$ of $\mat{A}$ is $\lambda_i$ in the column $\vect{v}_i$, so $\lambda_ic_{\vect{v}_i}=S$, and likewise $\lambda_ic_{\vect{u}_i}=S$. For odd $N$ the row $\vect{o}$ is zero, hence $S=0$; then $c_{\vect{u}_i}=c_{\vect{v}_i}=0$, and $c_{\vect{o}}=0$ follows from $\sum_ac_a=0$. For even $N$ we get $0=\sum_ac_a=2S\sum_i\lambda_i^{-1}$, hence $S=0$, and again all $c_a$ vanish.
\end{proof}

\begin{remark}\label{rem:null}
For $\varepsilon=0$ the equality $\pdim_0(\mat{D})=\rank(\mat{P}\mat{D}\mat{P})$ is the known formula for the exact dimension~\cite{Gower1985}; see also~\cite[Theorem~3.1]{NSS2023}. Theorem~\ref{thm:null} only adds that for $\mat{D}_N$ a relative error of any size $\varepsilon<1$ does not lower the dimension: the proof uses nothing but the exact preservation of zeros and the positivity of the $\lambda_i$. The lower bound holds for every signature of the target space. Metrics have no null squared distances, and for them the picture is different (Theorem~\ref{thm:main}).
\end{remark}

\subsection{Subgraphs of high-girth graphs}\label{sec:girth}

The girth of a graph is the length of its shortest cycle, and $+\infty$ for a forest. Let $G=(V,E)$ be a graph with girth greater than an integer $\ell\ge2$. Every simple graph qualifies for $\ell=2$. For a spanning subgraph $H\subseteq G$ define
\[
d_H(u,v)=\min\{\ell,\operatorname{dist}_H(u,v)\},
\]
where $\operatorname{dist}_H$ is the path metric of $H$, infinite between different components. This family of metrics is the one used in~\cite[Proposition~15.3.3]{MatousekBook}.

\begin{lemma}\label{lem:edges}
The function $d_H$ is a metric with values in $\{1,\dots,\ell\}$ on pairs of distinct vertices. For every edge $uv\in E$ we have $d_H(u,v)=1$ if $uv\in E(H)$ and $d_H(u,v)=\ell$ otherwise.
\end{lemma}
\begin{proof}
The minimum of an extended metric and a positive constant satisfies the triangle inequality. If $uv\in E\setminus E(H)$ and $H$ contains a path from $u$ to $v$ of length $L$, then this path together with the edge $uv$ is a cycle of length $L+1$ in $G$, so $L+1\ge\ell+1$.
\end{proof}

The condition on the girth cannot be weakened: in a cycle of length $\ell$ the removal of an edge leaves its endpoints at distance $\ell-1$.

Put $\varepsilon_\ell=(\ell^2-1)/(\ell^2+1)$. The intervals $[1-\varepsilon,1+\varepsilon]$ and $[\ell^2(1-\varepsilon),\ell^2(1+\varepsilon)]$ are disjoint exactly when $\varepsilon<\varepsilon_\ell$, that is, when the distortion $\sqrt{(1+\varepsilon)/(1-\varepsilon)}$ is less than $\ell$. For $\ell=2,3,5,7$ the thresholds are $3/5$, $4/5$, $12/13$, $24/25$.

\begin{theorem}\label{thm:girth}
Let $G$ be a graph with $N\ge2$ vertices, $m$ edges and girth greater than $\ell\ge2$, let $0\le\varepsilon<\varepsilon_\ell$, and let $s\ge1$ be an integer with $8Ns\le m$. Then at most $2^{\alpha m}$ of the $2^m$ subgraphs $H\subseteq G$ satisfy $\pdim_\varepsilon(d_H)\le s$, where
\[
\alpha=\frac18\log_2(64e)+\frac1{16}<1 .
\] In particular, there is $H\subseteq G$ with
\[
\pdim_\varepsilon(d_H)\ \ge\ \Bigl\lfloor\frac{m}{8N}\Bigr\rfloor+1 .
\]
\end{theorem}
\begin{proof}
Put $T=\frac12\bigl(1+\varepsilon+\ell^2(1-\varepsilon)\bigr)$, so that $1+\varepsilon<T<\ell^2(1-\varepsilon)$. A representation in $\R^{p,q}$ with $p+q\le s$ becomes one in $\R^{p,s-p}$ after appending zero coordinates, so it suffices to consider the $s+1$ signatures with $p+q=s$.

Fix such a signature. The coordinates of $N$ points are $Ns$ real variables, and for each edge $uv\in E$ the function $P_{uv}=\mathcal{Q}(\vect{x}_u-\vect{x}_v)-T$ is a polynomial of degree $2$ in them. If the points form an $\varepsilon$-representation of $d_H$, then by Lemma~\ref{lem:edges} we have $P_{uv}<0$ for $uv\in E(H)$ and $P_{uv}>0$ for $uv\in E\setminus E(H)$. No $P_{uv}$ vanishes, and the sign vector determines $H$. Since $m\ge Ns$, Lemma~\ref{lem:warren} bounds the number of such $H$ for one signature by $(8em/(Ns))^{Ns}$.

Put $x=m/(Ns)\ge8$. The total number of subgraphs with $\pdim_\varepsilon(d_H)\le s$ has binary logarithm at most
\[
\log_2(s+1)+m\,\frac{\log_2(8ex)}{x}.
\]
The function $x\mapsto\log_2(8ex)/x$ is decreasing for $x\ge1$, so for $x\ge8$ it is at most $\log_2(64e)/8$. Moreover, $\log_2(s+1)\le s\le m/(8N)\le m/16$, so the sum is at most $\alpha m$. The inequality $\alpha<1$ is equivalent to $\log_2(64e)<15/2$, that is, to $e<2\sqrt2$, which holds because $(2\sqrt2)^2=8>e^2$. The last assertion follows with $s=\lfloor m/(8N)\rfloor$ when this number is at least $1$, and is trivial otherwise.
\end{proof}

The coefficient $1/8$ cannot be replaced by $1/7$ in this argument: at $x=7$ the binary logarithm $\log_2(56e)$ exceeds $7$, because $56e>128$.

\begin{corollary}\label{cor:lower}
\leavevmode
\begin{enumerate}
\item For every $N\ge2$ and $0\le\varepsilon<3/5$ there is a metric on $N$ points with values in $\{1,2\}$ and $\pdim_\varepsilon>(N-1)/16$.
\item For every $N\ge2$ and $0\le\varepsilon<4/5$ there is a metric on $N$ points with values in $\{1,2,3\}$ and $\pdim_\varepsilon>(N-1)/32$.
\item For every $N\ge2$ and $0\le\varepsilon<12/13$ there is a metric on $N$ points with values in $\{1,\dots,5\}$ and $\pdim_\varepsilon>\sqrt{N/8}/64$.
\item For every prime power $q$ and $0\le\varepsilon<24/25$ there is a metric on $N=2(q+1)(q^2+1)$ points with values in $\{1,\dots,7\}$ and $\pdim_\varepsilon>(q+1)/16$, which is of order $N^{1/3}$.
\item For every integer $\ell\ge3$, every $N\ge2$ and $0\le\varepsilon<\varepsilon_\ell$ there is a metric on $N$ points with values in $\{1,\dots,\ell\}$ and $\pdim_\varepsilon>N^{1/(\ell-1)}/72$.
\end{enumerate}
\end{corollary}
\begin{proof}
We apply Theorem~\ref{thm:girth} and use $\lfloor t\rfloor+1>t$.

(1) Take $G=K_N$, $\ell=2$, $m=N(N-1)/2$.

(2) Take $G=K_{\lfloor N/2\rfloor,\lceil N/2\rceil}$, $\ell=3$, $m=\lfloor N^2/4\rfloor\ge(N^2-1)/4$, so that $m/(8N)\ge(N^2-1)/(32N)>(N-1)/32$.

(3) Let $q$ be a prime power. The points and the lines of the projective plane over $\mathbb{F}_q$ are the subspaces of $\mathbb{F}_q^3$ of dimension $1$ and $2$. There are $q^2+q+1$ of each, and every line contains $q+1$ points. Two points lie on exactly one line, so the bipartite incidence graph has no cycle of length $4$, and its girth is at least $6$. It has $2(q^2+q+1)$ vertices and $(q+1)(q^2+q+1)$ edges. For $N\ge14$ let $q$ be the largest power of $2$ with $2(q^2+q+1)\le N$. Adding isolated vertices to the incidence graph gives a graph $G$ on $N$ vertices. This does not change the girth, and Lemma~\ref{lem:edges} still applies. By the choice of $q$ we have $N<2((2q)^2+2q+1)<8(q^2+q+1)$. With $\ell=5$ we get $m/(8N)>(q+1)/64$, and $(q+1)^2>q^2+q+1>N/8$ gives $q+1>\sqrt{N/8}$. For $N<14$ the right-hand side is less than $1$, and every metric with values in $\{1,\dots,5\}$ qualifies.

(4) Let $q$ be a prime power and fix a non-degenerate alternating bilinear form on $\mathbb{F}_q^4$. Take as points the $(q+1)(q^2+1)$ one-dimensional subspaces and as lines the totally isotropic two-dimensional subspaces. Every line contains $q+1$ points. Every point $p$ lies on $q+1$ lines, one for each one-dimensional subspace of the two-dimensional space $p^\perp/p$. Hence there are also $(q+1)(q^2+1)$ lines. Two points lie on at most one line, and there are no three pairwise orthogonal points outside a common line, since they would span a totally isotropic subspace of dimension $3$. A cycle of length $6$ in the incidence graph would consist of three points and three distinct lines, each line containing two of the points. Such points are pairwise orthogonal and do not lie on a common line, because the three lines are distinct. So the bipartite incidence graph has no cycles of length $4$ or $6$, and its girth is at least $8$. It has $2(q+1)(q^2+1)$ vertices and $(q+1)^2(q^2+1)$ edges. Take $\ell=7$.

(5) By~\cite[Lemma~15.3.2]{MatousekBook} there is a graph on $N$ vertices with girth greater than $\ell$ and at least $\frac19N^{1+1/(\ell-1)}$ edges, so $m/(8N)\ge N^{1/(\ell-1)}/72$.
\end{proof}

For a fixed $\varepsilon<1$ the integer $\ell(\varepsilon)=\lfloor\sqrt{(1+\varepsilon)/(1-\varepsilon)}\rfloor+1$ satisfies $\varepsilon<\varepsilon_{\ell(\varepsilon)}$, so item~(5), or item~(1) when $\ell(\varepsilon)=2$, gives metrics with $\pdim_\varepsilon>N^{c'(\varepsilon)}/72$, $c'(\varepsilon)=1/(\ell(\varepsilon)-1)$. Up to the values of the constants, items~(1) and~(2) also follow from Proposition~\ref{prop:rank} and the theorem of Alon, Frankl and R\"odl~\cite[Corollary~1.2]{Alon1985} on the existence of sign matrices of order $n$ whose sign-rank is linear in $n$; see Remark~\ref{rem:signrank}. The constants in items~(3) and~(4) are explicit but are not optimized: these bounds serve to establish the asymptotic orders in the corresponding ranges. The constructions are based on the incidence graphs of projective planes and symplectic generalized quadrangles.

\begin{remark}\label{rem:counting-limit}
Theorem~\ref{thm:girth} yields at most $O(m/N)$, because it only uses the number $2^m$ of subgraphs. The orders $\sqrt N$ and $N^{1/3}$ in items~(3) and~(4) therefore reflect the largest known numbers of edges of graphs of girth $6$ and $8$.
\end{remark}

The linear bounds of items~(1) and~(2) come from a standard count: the family of metrics and the comparison of its size with the number of sign patterns are those of Matou\v{s}ek and of Alon, Frankl and R\"odl, and the only new point is that the count is carried over to relative error in $\R^{p,q}$. The metrics of these items are not given explicitly; explicit metrics with linear bounds are constructed in Section~\ref{sec:explicit}.

\begin{remark}\label{rem:two-valued}
Theorem~\ref{thm:girth} uses the triangle inequality only to make $d_H$ a metric. For symmetric matrices $\mat{D}$ with $\mat{D}_{uv}\in\{1,M^2\}$ for $u\ne v$, where $M>1$, the same count with $G=K_N$ and the threshold $T=\frac12\bigl(1+\varepsilon+M^2(1-\varepsilon)\bigr)$ gives a matrix with $\pdim_\varepsilon(\mat{D})>(N-1)/16$ for $0\le\varepsilon<\varepsilon_M$, where $\varepsilon_M=(M^2-1)/(M^2+1)$. The proof of Theorem~\ref{thm:log-upper} uses only that all off-diagonal entries lie in $[1,M^2]$, and for $\varepsilon>\varepsilon_M$ it gives $\pdim_\varepsilon(\mat{D})\le\lceil64\ln N/((M^2+1)^2(\varepsilon-\varepsilon_M)^2)\rceil$ for every such matrix. For two-valued matrices without the triangle inequality, the threshold $\varepsilon_M$ therefore separates the linear and the logarithmic regimes for every $M$. The thresholds $3/5$ and $4/5$ for metrics come from the triangle inequality, which bounds the ratio of the two values by $2$ when all distances are chosen freely and by $3$ in the bipartite construction of Section~\ref{sec:explicit}.
\end{remark}

\subsection{Explicit metrics}\label{sec:explicit}

The metrics of Corollary~\ref{cor:lower} exist but are not given explicitly. Explicit examples are obtained by reducing the dimension to a rank and applying known bounds on the approximate rank.

\begin{lemma}\label{lem:shift}
If points $\vect{x}_u\in\R^{p,q}$ have the matrix of squared intervals $\mat{A}$, then $\rank(\mat{A}-T\mat{J})\le p+q+2$ for every $T\in\R$; in particular, this holds for every submatrix of $\mat{A}-T\mat{J}$.
\end{lemma}
\begin{proof}
Let $\mat{G}$ be the Gram matrix of the points with respect to the bilinear form of $\mathcal{Q}$, and let $\vect{g}$ be its diagonal. Then $\rank\mat{G}\le p+q$ and $\mat{A}-T\mat{J}=(\vect{g}-T\vect{1})\vect{1}^{\mathsf T}+\vect{1}\vect{g}^{\mathsf T}-2\mat{G}$.
\end{proof}

A Hadamard matrix of order $n$ is a matrix $\mat{H}\in\{\pm1\}^{n\times n}$ with $\mat{H}\mat{H}^{\mathsf T}=n\mat{I}$. Let $M>1$. On $2n$ points $L\cup R$, $L=\{u_1,\dots,u_n\}$, $R=\{w_1,\dots,w_n\}$, consider the symmetric matrix $\mat{D}^M_{\mat{H}}$ with zero diagonal that equals $4$ on pairs of distinct points of the same part, $1$ on the pairs $(u_i,w_j)$ with $\mat{H}_{ij}=-1$, and $M^2$ on the pairs with $\mat{H}_{ij}=+1$. For $1\le M\le3$ it is the matrix of squared distances of a metric $d^M_{\mat{H}}$: triples inside one part have sides $(2,2,2)$, and a triple of two points of one part and one point of the other has sides $(2,h_1,h_2)$ with $h_1,h_2\in\{1,M\}$, for which $h_1+h_2\ge2$ and $|h_1-h_2|\le M-1\le2$. For $M>3$ and $n\ge2$ it is not a metric: some column of a Hadamard matrix contains entries of both signs, and the corresponding triple has sides $(2,M,1)$.

\begin{proposition}\label{prop:hadamard}
Let $M>1$, $\varepsilon_M=(M^2-1)/(M^2+1)$, $0\le\varepsilon<\varepsilon_M$, and
\[
a=\tfrac12\bigl(M^2(1-\varepsilon)-1-\varepsilon\bigr),\qquad b=\tfrac12\bigl(M^2(1+3\varepsilon)-1-\varepsilon\bigr),\qquad \alpha=\frac{b-a}{a+b}=\frac{2\varepsilon M^2}{(M^2-1)(1+\varepsilon)}.
\]
Then $0\le\alpha<1$ and
\[
\pdim_\varepsilon(\mat{D}^M_{\mat{H}})\ \ge\ (1-\alpha^2)\,n-2=\frac{4ab}{(a+b)^2}\,n-2 .
\]
\end{proposition}
\begin{proof}
The inequality $\alpha<1$ is equivalent to $\varepsilon(M^2+1)<M^2-1$. Let points in $\R^{p,q}$ have a matrix of squared intervals $\mat{A}\in\mathcal{A}_\varepsilon(\mat{D}^M_{\mat{H}})$, let $T=\frac12\bigl(1+\varepsilon+M^2(1-\varepsilon)\bigr)$, and let $\mat{W}$ be the block of $\mat{A}-T\mat{J}$ on the pairs $(u_i,w_j)$. By Lemma~\ref{lem:shift}, $\rank\mat{W}\le p+q+2$. If $\mat{H}_{ij}=-1$, then $\mat{A}_{u_iw_j}\in[1-\varepsilon,1+\varepsilon]$ and $-\mat{W}_{ij}\in[T-1-\varepsilon,\,T-1+\varepsilon]$; if $\mat{H}_{ij}=+1$, then $\mat{W}_{ij}\in[M^2(1-\varepsilon)-T,\,M^2(1+\varepsilon)-T]$. Since $T-1-\varepsilon=M^2(1-\varepsilon)-T=a>0$, $M^2(1+\varepsilon)-T=b$ and $T-1+\varepsilon=b-2\varepsilon(M^2-1)\le b$, in both cases $\mat{H}_{ij}\mat{W}_{ij}\in[a,b]$. Hence the matrix $\frac{2}{a+b}\mat{W}$ differs from $\mat{H}$ by at most $\alpha$ in every entry. By the theorem of Alon, Lee, Shraibman and Vempala~\cite[Theorem~1.2]{AlonApproxRank}, every matrix that differs entrywise by at most $\alpha\in(0,1)$ from a Hadamard matrix of order $n$ has rank at least $(1-\alpha^2)n$; for $\alpha=0$ this holds because $\mat{H}$ is nonsingular. Therefore $p+q+2\ge(1-\alpha^2)n$, and $1-\alpha^2=4ab/(a+b)^2$.
\end{proof}

\begin{corollary}\label{cor:hadamard}
For every $N\ge2$ there is an explicit metric on $N$ points with values in $\{1,2,3\}$ such that for every $0\le\varepsilon<4/5$
\[
\pdim_\varepsilon>\frac{(4-5\varepsilon)(4+13\varepsilon)}{64(1+\varepsilon)^2}\,N-2,
\]
and an explicit metric on $N$ points with values in $\{1,2\}$ such that for every $0\le\varepsilon<3/5$
\[
\pdim_\varepsilon>\frac{(3-5\varepsilon)(3+11\varepsilon)}{36(1+\varepsilon)^2}\,N-2 .
\]
\end{corollary}
\begin{proof}
Let $n=2^j$ be the largest power of two with $2n\le N$, so that $n>N/4$, and let $\mat{H}=\mat{H}_2^{\otimes j}$ be the Sylvester matrix of order $n$, $\mat{H}_2=\bigl(\begin{smallmatrix}1&1\\1&-1\end{smallmatrix}\bigr)$. For $M=3$ we have $\varepsilon_3=4/5$, $a=4-5\varepsilon$, $b=4+13\varepsilon$, $a+b=8(1+\varepsilon)$; for $M=2$ we have $\varepsilon_2=3/5$, $a=\frac12(3-5\varepsilon)$, $b=\frac12(3+11\varepsilon)$, $a+b=3(1+\varepsilon)$. Add to the $2n$ points of the metric $d^M_{\mat{H}}$ another $N-2n$ points at distance $2$ from all other points. Triples containing new points have sides $(2,2,h)$ with $h\in\{1,2,M\}$, so the result is a metric, and an $\varepsilon$-representation of it restricts to an $\varepsilon$-representation of $d^M_{\mat{H}}$. By Proposition~\ref{prop:hadamard}, $\pdim_\varepsilon\ge4abn/(a+b)^2-2>abN/(a+b)^2-2$.
\end{proof}

The same device applies to the metric of the Paley graph, which is not bipartite. Let $N\equiv1\pmod4$ be a prime power and $\chi$ the quadratic character of the field $\mathbb{F}_N$, $\chi(0)=0$. The vertices of the Paley graph are the elements of $\mathbb{F}_N$, and $x,y$ are adjacent if $\chi(x-y)=1$. Its Seidel matrix $\mat{C}=(-\chi(x-y))_{x,y\in\mathbb{F}_N}$ is symmetric, has zero diagonal, equals $-1$ on the edges and $1$ on the other pairs, and satisfies $\mat{C}^2=N\mat{I}-\mat{J}$. Indeed, $\chi(-1)=1$ because $N\equiv1\pmod4$, so $(\mat{C}^2)_{xz}=\sum_y\chi(y)\chi(y+h)$ with $h=z-x$. For $h=0$ the sum is $N-1$. For $h\ne0$ and $y\ne0$ we have $\chi(y)\chi(y+h)=\chi(1+hy^{-1})$, and $1+hy^{-1}$ runs over all elements of the field except $1$; since $\sum_t\chi(t)=0$, the sum equals $-1$. Hence the spectral norm of $\mat{C}$ is $\|\mat{C}\|=\sqrt N$.

\begin{proposition}\label{prop:paley}
Let $N\equiv1\pmod4$ be a prime power, let $d$ be the metric on $\mathbb{F}_N$ that equals $1$ on the edges of the Paley graph and $2$ on the other pairs, and let $0\le\varepsilon<3/5$. Then
\[
\pdim_\varepsilon(d)\ \ge\ \frac{(3-5\varepsilon)(3+11\varepsilon)(N-1)-(5-3\varepsilon)^2}{9(1+\varepsilon)^2}-2 .
\]
\end{proposition}
\begin{proof}
Let points in $\R^{p,q}$ have a matrix of squared intervals $\mat{A}\in\mathcal{A}_\varepsilon(d)$, let $T=\frac12(5-3\varepsilon)$ and $\mat{W}=\mat{A}-T\mat{J}$. By Lemma~\ref{lem:shift}, $\rank\mat{W}\le p+q+2$, and $\mat{W}_{xx}=-T$. Put $a=\frac12(3-5\varepsilon)>0$ and $b=\frac12(3+11\varepsilon)$. On the edges $\mat{A}_{xy}\in[1-\varepsilon,1+\varepsilon]$ and $\mat{W}_{xy}\in[-a-2\varepsilon,-a]$; on the other pairs $\mat{A}_{xy}\in[4-4\varepsilon,4+4\varepsilon]$ and $\mat{W}_{xy}\in[a,b]$. Since $a+2\varepsilon\le b$, for $x\ne y$ we have $\sign\mat{W}_{xy}=\mat{C}_{xy}$ and $a\le|\mat{W}_{xy}|\le b$.

Let $\|\mat{W}\|_{\mathrm{tr}}$ be the sum of the singular values of $\mat{W}$ and $\|\mat{W}\|_F$ its Frobenius norm, and let $S=\sum_{x\ne y}|\mat{W}_{xy}|=\sum_{x,y}\mat{W}_{xy}\mat{C}_{xy}$. From $(t-a)(b-t)\ge0$ we get $t^2\le(a+b)t-ab$ for $t\in[a,b]$. Summing over the $N(N-1)$ ordered pairs $x\ne y$ and adding the $N$ diagonal entries gives
\[
\|\mat{W}\|_F^2\le(a+b)S-c_0,\qquad c_0=ab\,N(N-1)-NT^2 .
\]
By the duality of the trace and spectral norms, $S\le\|\mat{W}\|_{\mathrm{tr}}\|\mat{C}\|$, and $\|\mat{W}\|_{\mathrm{tr}}\le\sqrt{\rank\mat{W}}\,\|\mat{W}\|_F$ by the Cauchy--Schwarz inequality for the singular values; hence $S^2\le N\rank(\mat{W})\|\mat{W}\|_F^2$. If $c_0\le0$, the right-hand side of the claimed inequality is at most $-2$. Let $c_0>0$. From $((a+b)S-2c_0)^2\ge0$ we get $(a+b)^2S^2\ge4c_0((a+b)S-c_0)$, and $(a+b)S-c_0\ge\|\mat{W}\|_F^2>0$. Therefore
\[
\rank\mat{W}\ \ge\ \frac{S^2}{N\,((a+b)S-c_0)}\ \ge\ \frac{4c_0}{N(a+b)^2}=\frac{4ab(N-1)-4T^2}{(a+b)^2}.
\]
It remains to substitute $4ab=(3-5\varepsilon)(3+11\varepsilon)$, $(a+b)^2=9(1+\varepsilon)^2$ and $4T^2=(5-3\varepsilon)^2$.
\end{proof}

The rank bounds in Propositions~\ref{prop:hadamard} and~\ref{prop:paley} are known. The inequality $\rank\mat{W}\ge S^2/(\|\mat{C}\|^2\|\mat{W}\|_F^2)$ is the standard trace-norm bound on the approximate rank, cf.~\cite[Prop.~2.2, 2.3]{Sherstov2011}; the quadratic estimate of $\|\mat{W}\|_F^2$ sharpens the constant, and for a Hadamard matrix the same argument gives exactly the bound $(1-\alpha^2)n$ of~\cite[Theorem~1.2]{AlonApproxRank}. What is new here is only the reduction of the dimension to a rank (Lemma~\ref{lem:shift}) and the choice of configurations. The coefficient of $N$ in Proposition~\ref{prop:paley} tends to $1$ as $\varepsilon\to0$ and to $0$ as $\varepsilon\to3/5$. Bilu and Linial~\cite{BiluLinial2005} conjecture that the sphericity of the Paley graph is linear; Proposition~\ref{prop:paley} does not bear on this conjecture, because an $\varepsilon$-representation provides the gap $a\le|\mat{W}_{xy}|\le b$, which the definition of sphericity does not.

\begin{remark}\label{rem:signrank}
Lemma~\ref{lem:shift} also gives a bound through the sign-rank. If the squared distances between the two parts of a bipartite matrix take the values $1$ and $M^2$ according to the signs of a matrix $\mat{S}\in\{\pm1\}^{n\times n}$, and $\varepsilon<\varepsilon_M$, then the signs of the block $\mat{W}$ in the proof of Proposition~\ref{prop:hadamard} are those of $\mat{S}$, and $\pdim_\varepsilon\ge\operatorname{signrank}(\mat{S})-2$. Combined with the theorem of Alon, Frankl and R\"odl on the existence of sign matrices of order $n$ with sign-rank at least $n/32$~\cite[Corollary~1.2 and Sect.~6]{Alon1985}, this gives another proof of item~(2) of Corollary~\ref{cor:lower} with a worse constant. For a Hadamard matrix, Forster's theorem~\cite{Forster2002} gives only $\operatorname{signrank}(\mat{H})\ge\sqrt n$, and for the Sylvester matrices of Corollary~\ref{cor:hadamard} the sign-rank cannot give more. The matrix of order $4$ with ones on the diagonal and $-1$ elsewhere has sign-rank $3$~\cite[Claim~3]{AlonMoranYehudayoff2016} and is obtained from $\mat{H}_2^{\otimes2}$ by permuting rows and columns and changing their signs. Hence the sign-rank of the Sylvester matrix of order $4^j$ is at most $3^j$, as noted in~\cite[Sect.~4]{AlonMoranYehudayoff2016}. The linear bound of Proposition~\ref{prop:hadamard} uses, besides the signs, the two-sided bounds on $|\mat{W}_{ij}|$.
\end{remark}

\section{The high-error range}\label{sec:high}

For metrics with distances in $[1,\Phi]$ let $\varepsilon_\Phi=(\Phi^2-1)/(\Phi^2+1)$. For $\varepsilon>\varepsilon_\Phi$ a single value of the squared interval is admissible for all pairs, and the dimension becomes logarithmic.

\begin{theorem}\label{thm:log-upper}
Let $d$ be a metric on $N\ge2$ points with all distances in $[1,\Phi]$, $\Phi\ge1$, and let $\varepsilon_\Phi<\varepsilon<1$. Put $L=(1-\varepsilon)\Phi^2$, $U=1+\varepsilon$, $c=\frac12(L+U)$ and $g=U-L=(\Phi^2+1)(\varepsilon-\varepsilon_\Phi)>0$. Then $d$ has an $\varepsilon$-representation in the Euclidean space of dimension
\[
\Bigl\lceil\frac{16\,c^2\ln N}{g^2}\Bigr\rceil\ \le\ \Bigl\lceil\frac{64\ln N}{(\Phi^2+1)^2(\varepsilon-\varepsilon_\Phi)^2}\Bigr\rceil .
\]
\end{theorem}
\begin{proof}
If $L\le y\le U$ and $1\le t\le\Phi^2$, then $(1-\varepsilon)t\le(1-\varepsilon)\Phi^2=L\le y\le U=1+\varepsilon\le(1+\varepsilon)t$. So it suffices to find points in $\R^n$ all of whose squared distances lie in $[L,U]=[c-g/2,c+g/2]$.

Let $n=\lceil16c^2\ln N/g^2\rceil$ and choose $\vect{x}_1,\dots,\vect{x}_N\in\{-\sqrt{c/(2n)},\sqrt{c/(2n)}\}^n$ with independent uniform coordinates. For $i\ne j$ the vectors differ in $H_{ij}$ coordinates, where $H_{ij}$ is a sum of $n$ independent variables with values in $\{0,1\}$ and mean $1/2$, and $\|\vect{x}_i-\vect{x}_j\|^2=(2c/n)H_{ij}$. By Hoeffding's inequality~\cite[Theorem~2]{Hoeffding1963}
\[
\Pr\Bigl[\bigl|\|\vect{x}_i-\vect{x}_j\|^2-c\bigr|>\frac g2\Bigr]=\Pr\Bigl[\Bigl|H_{ij}-\frac n2\Bigr|>\frac{g}{4c}\,n\Bigr]\le2\exp\Bigl(-\frac{g^2n}{8c^2}\Bigr)\le\frac2{N^2}.
\]
There are $N(N-1)/2$ pairs, so with probability at least $1/N$ all squared distances lie in $[L,U]$. Finally, $c\le U<2$.
\end{proof}

At $\varepsilon=\varepsilon_\Phi$ we have $g=0$, and the argument does not apply.

For the lower bound we use the following consequence of the method of Alon~\cite{Alon2009}. For $0<\varepsilon<1$ put $\lambda=\ln(1/\varepsilon)$ and
\[
\beta(\varepsilon)=\begin{cases}1,&\lambda\ge\ln(2e),\\ \lambda/\bigl(2\ln(2e/\lambda)\bigr),&\lambda<\ln(2e),\end{cases}\qquad c(\varepsilon)=\frac{\beta(\varepsilon)}{2\lambda}.
\]
Then $0<\beta\le1$ and $\beta\ln(2e/\beta)\le\lambda$: for $\lambda<\ln(2e)$ and $\Lambda=\ln(2e/\lambda)>1$ we have $\beta\ln(2e/\beta)=\frac{\lambda}{2\Lambda}(\Lambda+\ln(2\Lambda))\le\lambda$, because $\ln(2\Lambda)\le\Lambda$. For $\varepsilon>1/(2e)$ we get $c(\varepsilon)=1/(4\ln(2e/\ln(1/\varepsilon)))$.

\begin{lemma}\label{lem:perturbed}
Let $\mat{M}$ be a real symmetric matrix of order $n\ge4/\varepsilon^2$ with $\mat{M}_{ii}=1$ and $|\mat{M}_{ij}|\le\varepsilon$ for $i\ne j$, where $0<\varepsilon<1$. Then $\rank\mat{M}>1+c(\varepsilon)\ln n$.
\end{lemma}
\begin{proof}
For a symmetric matrix $\mat{A}$ with unit diagonal and off-diagonal entries of modulus at most $\theta$ the Cauchy--Schwarz inequality for the eigenvalues gives
\begin{equation}\label{eq:trace-rank}
\rank\mat{A}\ge\frac{(\operatorname{tr}\mat{A})^2}{\operatorname{tr}\mat{A}^2}\ge\frac{n}{1+(n-1)\theta^2}.
\end{equation}
Let $r=\rank\mat{M}$ and write $\mat{M}=\sum_{l\le r}\sigma_l\vect{v}_l\vect{v}_l^{\mathsf T}$. For an integer $t\ge1$ the entrywise power $\mat{M}^{\circ t}$ has entries $(\sum_l\sigma_lv_l(i)v_l(j))^t$, a sum of at most $\binom{r+t-1}{t}$ terms of the form $a\,m(\vect{v}(i))\,m(\vect{v}(j))$ with $m$ a monomial of degree $t$ in $r$ variables; so $\rank\mat{M}^{\circ t}\le\binom{r+t-1}{t}$. It is symmetric with unit diagonal and off-diagonal entries of modulus at most $\varepsilon^t$. Let $t=\lceil\ln n/(2\lambda)\rceil$, so that $\varepsilon^{2t}\le1/n$; by~\eqref{eq:trace-rank}
\[
\binom{r+t-1}{t}>\frac n2 .
\]
Let $m=r-1$; then $m\ge1$, because $\binom tt=1\le n/2$. Suppose $m\le\beta t$, where $\beta=\beta(\varepsilon)$. Then $m\le t$ and $\binom{m+t}{m}\le(e(m+t)/m)^m\le(2et/m)^m$. The function $m\mapsto m\ln(2et/m)$ increases on $(0,2t]$, hence
\[
\ln\frac n2<m\ln\frac{2et}{m}\le\beta t\ln\frac{2e}{\beta}\le\lambda t\le\frac{\ln n}2+\lambda ,
\]
that is, $n<4/\varepsilon^2$, a contradiction. Thus $r-1>\beta t\ge\beta\ln n/(2\lambda)$.
\end{proof}

For $1/\sqrt n\le\varepsilon<1/2$ the theorem of Alon~\cite[Theorem~2.1]{Alon2009} gives the stronger bound of order $\ln n/(\varepsilon^2\ln(1/\varepsilon))$, also for non-symmetric matrices.

\begin{proposition}\label{prop:log-lower}
Let $d$ be the uniform metric on $N$ points, $d(u,v)=1$ for $u\ne v$, and let $0<\varepsilon<1$, $N\ge4/\varepsilon^2$. Then $\pdim_\varepsilon(d)>c(\varepsilon)\ln N-2$.
\end{proposition}
\begin{proof}
Let $\mat{A}$ be the matrix of squared intervals of an $\varepsilon$-representation in dimension $s$. It is symmetric with zero diagonal, its off-diagonal entries lie in $[1-\varepsilon,1+\varepsilon]$, and $\rank\mat{A}\le s+2$, as in the proof of Proposition~\ref{prop:rank}. The matrix $\mat{M}=\mat{J}-\mat{A}$ satisfies the assumptions of Lemma~\ref{lem:perturbed} and has rank at most $s+3$.
\end{proof}

The argument uses only the rank of $\mat{A}$ and therefore applies to every signature. The uniform metric belongs to $\mathcal{M}_N(\Phi)$ for every $\Phi\ge1$, so Theorem~\ref{thm:log-upper} and Proposition~\ref{prop:log-lower} give $F_{\Phi,\varepsilon}(N)=\Theta_{\Phi,\varepsilon}(\log N)$ for $\varepsilon_\Phi<\varepsilon<1$.

\begin{proof}[Proof of Theorem~\ref{thm:main}]
The lower bounds of items~(1) and~(2) are items~(2) and~(3) of Corollary~\ref{cor:lower}; the metrics there have values in $\{1,\dots,5\}$ and belong to $\mathcal{M}_N(\Phi)$ for $\Phi\ge5$. The upper bounds are Proposition~\ref{prop:rank} and Corollary~\ref{cor:spread}. Item~(3) is Theorem~\ref{thm:log-upper} and Proposition~\ref{prop:log-lower}. The assumptions of the latter hold: $\varepsilon>\varepsilon_\Phi\ge12/13$ gives $4/\varepsilon^2<5\le N$ and $\varepsilon>1/(2e)$.
\end{proof}

\section{The parameter \texorpdfstring{$C_*$}{C*}}\label{sec:C}

Let $\mat{D}=(d(u,v)^2)$. An exact representation of $d$ consists of points $\vect{x}_u=(\vect{x}_u^+,\vect{x}_u^-)\in\R^p\oplus\R^q$ with $\|\vect{x}_u^+-\vect{x}_v^+\|^2-\|\vect{x}_u^--\vect{x}_v^-\|^2=\mat{D}_{uv}$. Put
\[
C(\vect{x})=\max_{u\ne v}\frac{\|\vect{x}_u^+-\vect{x}_v^+\|^2+\|\vect{x}_u^--\vect{x}_v^-\|^2}{\mat{D}_{uv}},\qquad C_*(d)=\inf_{\vect{x}}C(\vect{x}),
\]
the infimum being over all exact representations in all $\R^{p,q}$. The quantity $C(\vect{x})$ depends on the coordinates: a linear map preserving $\mathcal{Q}$ may change the Euclidean norm.

\begin{lemma}[{\cite[Lemma~2.2 and Theorem~2.3]{Deng2025}}, with a constant]\label{lem:deng}
If $d$ is a metric on $N$ points with an exact representation $\vect{x}$ and $0<\varepsilon<1$, then, with $\eta=\varepsilon/C(\vect{x})$,
\[
\pdim_\varepsilon(d)\le2\Bigl\lceil\frac{4\ln N}{\eta^2/2-\eta^3/3}\Bigr\rceil\le2\Bigl\lceil\frac{24\,C(\vect{x})^2\ln N}{\varepsilon^2}\Bigr\rceil .
\]
\end{lemma}
\begin{proof}
Apply Lemma~\ref{lem:jl} with $\eta=\varepsilon/C(\vect{x})$ separately to the points $\vect{x}_u^+$ and to the points $\vect{x}_u^-$; a part whose dimension is already at most $\lceil4\ln N/(\eta^2/2-\eta^3/3)\rceil$ is left unchanged. Here $0<\eta<1$ because $C(\vect{x})\ge1$. The squared interval changes by at most $\eta(\|\vect{x}_u^+-\vect{x}_v^+\|^2+\|\vect{x}_u^--\vect{x}_v^-\|^2)\le\varepsilon\mat{D}_{uv}$.
\end{proof}

Let $\mat{G}_c=-\frac12\mat{P}\mat{D}\mat{P}$, and let $\|\mat{M}\|_{\mathrm{tr}}$ denote the sum of the absolute values of the eigenvalues of a symmetric matrix $\mat{M}$.

\begin{proposition}\label{prop:Cstar}
For every metric $d$ on $N$ points the infimum $C_*(d)$ is attained, and
\[
C_*(d)=1+2\min\Bigl\{\max_{u\ne v}\frac{\mathcal{K}(\mat{Z})_{uv}}{\mat{D}_{uv}}:\ \mat{Z}\succeq0,\ \mat{Z}+\mat{G}_c\succeq0,\ \mat{Z}\vect{1}=\vect{0}\Bigr\}\ \ge\ \frac{\|\mat{G}_c\|_{\mathrm{tr}}}{\operatorname{tr}\mat{G}_c}.
\]
\end{proposition}
\begin{proof}
Let $\mat{Y}_0,\mat{Z}_0$ be the Gram matrices of the points $\vect{x}_u^+$ and $\vect{x}_u^-$. They are positive semidefinite, $\mathcal{K}(\mat{Y}_0)-\mathcal{K}(\mat{Z}_0)=\mat{D}$, and the numerator in $C(\vect{x})$ equals $\mat{D}_{uv}+2\mathcal{K}(\mat{Z}_0)_{uv}$. Replacing them by $\mat{P}\mat{Y}_0\mat{P}$ and $\mat{Z}=\mat{P}\mat{Z}_0\mat{P}$ does not change $\mathcal{K}$. The difference of the new matrices is centred and has $\mathcal{K}$-image $\mat{D}$, so it equals $\mat{G}_c$. Thus $\mat{Z}+\mat{G}_c\succeq0$. Conversely, every admissible $\mat{Z}$ defines a representation. For a centred matrix $\mat{M}$ we have $\sum_{u<v}\mathcal{K}(\mat{M})_{uv}=N\operatorname{tr}\mat{M}$; therefore the admissible $\mat{Z}$ with bounded objective have bounded trace, and the minimum is attained.

For the inequality let $\tau^+$ and $\tau^-$ be the sums of the absolute values of the positive and of the negative eigenvalues of $\mat{G}_c$, and let $\mat{\Pi}$ be the orthogonal projector onto the negative eigenspace. From $\mat{Z}\succeq-\mat{G}_c$ and $\mat{Z}\succeq0$ we get $\operatorname{tr}\mat{Z}\ge\operatorname{tr}(\mat{\Pi}\mat{Z}\mat{\Pi})\ge\tau^-$. Since a maximum of ratios is at least the ratio of the sums,
\[
\max_{u\ne v}\frac{\mathcal{K}(\mat{Z})_{uv}}{\mat{D}_{uv}}\ \ge\ \frac{\sum_{u<v}\mathcal{K}(\mat{Z})_{uv}}{\sum_{u<v}\mathcal{K}(\mat{G}_c)_{uv}}=\frac{\operatorname{tr}\mat{Z}}{\operatorname{tr}\mat{G}_c}\ \ge\ \frac{\tau^-}{\tau^+-\tau^-},
\]
and $1+2\tau^-/(\tau^+-\tau^-)=(\tau^++\tau^-)/(\tau^+-\tau^-)$.
\end{proof}

In particular, $C_*(d)=1$ exactly for Euclidean metrics. For a symmetric matrix $\mat{B}$ let $|\mat{B}|=(\mat{B}^2)^{1/2}$.

\begin{lemma}\label{lem:rows}
Let $\mat{B}$ be a symmetric matrix with $\mathcal{K}(\mat{B})=\mat{D}$. The representation obtained from the spectral decomposition of $\mat{B}$ satisfies
\[
C(\vect{x})\le\max_{u\ne v}\frac{2\bigl(\|\mat{B}\vect{e}_u\|+\|\mat{B}\vect{e}_v\|\bigr)}{\mat{D}_{uv}}.
\]
\end{lemma}
\begin{proof}
The Euclidean squared distances of this representation are $\mathcal{K}(|\mat{B}|)_{uv}$. Since $|\mat{B}|\succeq0$, we have $\mathcal{K}(|\mat{B}|)_{uv}\le(\sqrt{|\mat{B}|_{uu}}+\sqrt{|\mat{B}|_{vv}})^2\le2(|\mat{B}|_{uu}+|\mat{B}|_{vv})$, and by the Cauchy--Schwarz inequality $|\mat{B}|_{uu}=\langle\vect{e}_u,|\mat{B}|\vect{e}_u\rangle\le\||\mat{B}|\vect{e}_u\|=\|\mat{B}\vect{e}_u\|$.
\end{proof}

\begin{corollary}\label{cor:graph}
Let $H$ be a graph on $N$ vertices with maximum degree $\Delta$ and adjacency matrix $\mat{A}_H$.
\begin{enumerate}
\item The metric $d$ equal to $1$ on the edges of $H$ and to $2$ on the other pairs has $C_*(d)\le2\sqrt{16+9\Delta}$.
\item For $4/5\le\varepsilon<1$ the metric $d_H=\min\{5,\operatorname{dist}_H\}$ satisfies $\pdim_\varepsilon(d_H)=O(\Delta\log N)$ with an absolute constant.
\end{enumerate}
\end{corollary}
\begin{proof}
(1) For $\mat{B}=2\mat{I}+\frac32\mat{A}_H$ we have $\mathcal{K}(\mat{B})_{uv}=4-3(\mat{A}_H)_{uv}$ and $\|\mat{B}\vect{e}_u\|^2=4+\frac94\deg u$; apply Lemma~\ref{lem:rows}.

(2) Let $\mat{B}=3\mat{I}+\frac52\mat{A}_H$, so that $\mathcal{K}(\mat{B})$ equals $1$ on the edges and $6$ on the other pairs. By Lemma~\ref{lem:rows}, applied with $\mathcal{K}(\mat{B})$ in place of $\mat{D}$, its spectral representation $\vect{x}$ has $C(\vect{x})\le2\sqrt{36+25\Delta}$, the ratio being taken with respect to $\mathcal{K}(\mat{B})$. As in Lemma~\ref{lem:deng}, projecting with $\eta=1/(6C(\vect{x}))$ to dimension $2\lceil24\ln N/\eta^2\rceil\le2\lceil3456(36+25\Delta)\ln N\rceil$, where $3456=24\cdot6^2\cdot4$, gives points whose squared intervals lie in $[5/6,7/6]$ on the edges and in $[5,7]$ on the other pairs. For $\varepsilon\ge4/5$ the first interval is contained in $[1-\varepsilon,1+\varepsilon]$, and the second in $[25(1-\varepsilon),4(1+\varepsilon)]$, which is admissible for every distance between $2$ and $5$.
\end{proof}

Item~(2) shows that for $\ell=5$ and $4/5\le\varepsilon<12/13$ the lower bound of Theorem~\ref{thm:girth} is sharp on its own family up to a logarithmic factor: for a $\Delta$-regular graph $G$ of girth at least $6$ some subgraph requires dimension more than $\Delta/16$, and every subgraph admits dimension $O(\Delta\log N)$. Here the value $\varepsilon=4/5$ is included.

\begin{theorem}\label{thm:Cfamily}
For every $N\ge1000$ there is a graph $H$ on $N$ vertices such that the metric $d$ equal to $1$ on the edges of $H$ and to $2$ on the other pairs satisfies
\[
\frac{\sqrt N}{16}\le C_*(d)\le7\sqrt N\qquad\text{and}\qquad\pdim_\varepsilon(d)>\frac{N-1}{16}\quad\text{for every }\varepsilon\in(0,3/5).
\]
\end{theorem}
\begin{proof}
Choose $H$ uniformly among the $2^m$ graphs on $N$ vertices, $m=N(N-1)/2$. Then $\mat{D}=\frac52(\mat{J}-\mat{I})+\frac32\mat{S}$, where $\mat{S}$ is symmetric with zero diagonal and independent uniform signs above the diagonal. Consider three events.

$\mathcal{E}_1$: $\|\mat{S}\vect{1}\|^2\le8N^2$. Since $\mathbb{E}\|\mat{S}\vect{1}\|^2=N(N-1)$, Markov's inequality gives probability at least $7/8$.

$\mathcal{E}_2$: $\operatorname{tr}\mat{S}^4\le16N^3$. In $\operatorname{tr}\mat{S}^4=\sum\mat{S}_{uv}\mat{S}_{vw}\mat{S}_{wz}\mat{S}_{zu}$ the terms with non-zero expectation are those with $w=u$ or $z=v$; there are $2N(N-1)^2-N(N-1)$ of them, so $\mathbb{E}\operatorname{tr}\mat{S}^4=N(N-1)(2N-3)<2N^3$ and the probability is at least $7/8$.

$\mathcal{E}_3$: $\pdim_\varepsilon(d)>(N-1)/16$ for all $\varepsilon<3/5$. By Theorem~\ref{thm:girth} with $G=K_N$, $\ell=2$ and $s=\lfloor(N-1)/16\rfloor$, for each $\varepsilon<3/5$ at most $2^{\alpha m}$ graphs have $\pdim_\varepsilon\le s$. These sets of graphs increase with $\varepsilon$, so their union over $\varepsilon<3/5$ also has at most $2^{\alpha m}$ elements, and $\mathcal{E}_3$ fails with probability at most $2^{-(1-\alpha)m}$. Here $1-\alpha=\frac1{16}\log_2(8/e^2)$, and $8/e^2>21/20$ because $e^2<7.6$; since $\log_2(1+x)\ge\ln(1+x)\ge x/(1+x)$, we get $1-\alpha>1/336$. As $m\ge499\,500$, the probability is less than $1/4$.

Fix $H$ in the intersection of the three events. We have $\mat{G}_c=\frac54\mat{P}-\frac34\mat{M}$ with $\mat{M}=\mat{P}\mat{S}\mat{P}$. H\"older's inequality for the eigenvalues gives $\|\mat{M}\|_F^2\le\|\mat{M}\|_{\mathrm{tr}}^{2/3}(\operatorname{tr}\mat{M}^4)^{1/3}$, and $\operatorname{tr}\mat{M}^4\le\operatorname{tr}\mat{S}^4\le16N^3$, because the singular values of the compression $\mat{P}\mat{S}\mat{P}$ do not exceed those of $\mat{S}$. Moreover,
\[
\|\mat{M}\|_F^2=\|\mat{S}\|_F^2-\tfrac2N\|\mat{S}\vect{1}\|^2+\tfrac1{N^2}(\vect{1}^{\mathsf T}\mat{S}\vect{1})^2\ \ge\ N(N-1)-16N .
\]
Hence $\|\mat{M}\|_{\mathrm{tr}}\ge\|\mat{M}\|_F^3/(\operatorname{tr}\mat{M}^4)^{1/2}\ge\frac14N^{3/2}(1-17/N)^{3/2}$ and $\|\mat{G}_c\|_{\mathrm{tr}}\ge\frac34\|\mat{M}\|_{\mathrm{tr}}-\frac54(N-1)$. Since $\operatorname{tr}\mat{G}_c=\frac1N\sum_{u<v}\mat{D}_{uv}<2N$, Proposition~\ref{prop:Cstar} yields
\[
C_*(d)\ \ge\ \frac3{32}\sqrt N\Bigl(1-\frac{17}N\Bigr)^{3/2}-\frac58\ \ge\ \frac{\sqrt N}{16}\qquad(N\ge1000).
\]
The upper bound is Corollary~\ref{cor:graph}(1) with $\Delta\le N-1$: $C_*(d)\le2\sqrt{9N+7}\le7\sqrt N$.
\end{proof}

Since $(N-1)/16\ge N/16-1$ and $N\ge C_*^2/7^2$, the metrics of Theorem~\ref{thm:Cfamily} satisfy, together with Lemma~\ref{lem:deng},
\[
\frac{C_*^2}{784}-1\ \le\ \pdim_\varepsilon(d)\ \le\ 2\Bigl\lceil\frac{24\,C_*^2\ln N}{\varepsilon^2}\Bigr\rceil\qquad(\varepsilon<3/5).
\]
So if a bound $\pdim_\varepsilon\le K\,\varphi(C_*)\ln N$ with a non-decreasing function $\varphi$ holds for all metrics and some fixed $\varepsilon<3/5$, then $\varphi(7\sqrt N)\ge(N/16-1)/(K\ln N)$ for all $N\ge1000$, and by monotonicity $\varphi(C)\ge cC^2/\ln C$ for all sufficiently large $C$, with $c>0$ depending on $K$. In particular, $\varphi(C)=C^a$ requires $a\ge2$. Only the upper bound $C_*\le7\sqrt N$ is used here. The factor $\varepsilon^{-2}\ln N$ between the two sides remains, and on this family the upper bound exceeds the trivial $N-1$. The statement is restricted to $\varepsilon<3/5$: for $\varepsilon>3/5$ a metric with values $1$ and $2$ is approximated by a multiple of $\mat{J}-\mat{I}$ and has $\pdim_\varepsilon=O(\log N)$ whatever its $C_*$ is. The conclusion on the exponent relies on the upper bound $C_*\le7\sqrt N$ of Theorem~\ref{thm:Cfamily}.

At the other end of the scale, the uniform metric is Euclidean and has $C_*=1$. Lemma~\ref{lem:deng} gives $\pdim_\varepsilon\le2\lceil24\ln N/\varepsilon^2\rceil$, while Proposition~\ref{prop:log-lower} gives $\pdim_\varepsilon>c(\varepsilon)\ln N-2$. For $1/\sqrt N\le\varepsilon<1/2$, \cite[Theorem~2.1]{Alon2009} improves the lower bound to order $\ln N/(\varepsilon^2\ln(1/\varepsilon))$. So the factor $\ln N$ in the bound of Lemma~\ref{lem:deng} is needed when $C$ is bounded.

The argument of Theorem~\ref{thm:Cfamily} works at every scale. Let $N\ge2$, let $\Delta\le N-1$ with $N\Delta$ even, and let $0\le\varepsilon<3/5$. Applying Theorem~\ref{thm:girth} with $\ell=2$ to a $\Delta$-regular graph and using Corollary~\ref{cor:graph}(1), we obtain a graph of maximum degree at most $\Delta$ whose two-valued metric has $\pdim_\varepsilon>\Delta/16$ and $C_*\le2\sqrt{16+9\Delta}$.

\begin{remark}
The quantity $\max_u(\mat{Y}'+\mat{Z}')_{uu}$ minimized over decompositions $\mat{B}=\mat{Y}'-\mat{Z}'$ into positive semidefinite matrices is the factorization norm $\gamma_2(\mat{B})$ of a symmetric matrix, and Lemmas~\ref{lem:deng} and~\ref{lem:rows} are counterparts of the bounds relating $\gamma_2$ to the approximate rank of sign matrices~\cite{LMSS2007,LeeShraibman2009,AlonApproxRank}. In particular, the mechanism of Lemma~\ref{lem:deng} is that of~\cite[Corollary~1.5]{AlonApproxRank}.
\end{remark}

\section{Discussion and open questions}\label{sec:discussion}

\paragraph{Summary.} Theorem~\ref{thm:upper} bounds the dimension of an $\varepsilon$-representation of an arbitrary metric for $\varepsilon>4/5$ in terms of the local ratios $R_k(u)/\nu(u)$. For metrics with distances in $[1,\Phi]$, $\Phi\ge5$, it implies the upper bound in the middle range of Theorem~\ref{thm:main}. By that theorem the worst-case dimension is linear in $N$ for $\varepsilon<4/5$, of order $\sqrt N$ up to a factor $O(\log N)$ for $4/5<\varepsilon<12/13$, and of order $\log N$ for $\varepsilon>(\Phi^2-1)/(\Phi^2+1)$. For all metrics we have the lower bounds of Corollary~\ref{cor:lower}, and Corollary~\ref{cor:hadamard} and Proposition~\ref{prop:paley} give explicit linear examples for $\varepsilon<4/5$. For configurations with null intervals the worst-case dimension is $N-1$ for every $\varepsilon<1$ (Theorem~\ref{thm:null}). The parameter $C_*$ controls the dimension from above by Lemma~\ref{lem:deng}, and its square is a lower bound up to constants on the family of Theorem~\ref{thm:Cfamily} for $\varepsilon<3/5$.

\paragraph{Open questions.}
\begin{enumerate}
\item \emph{The boundary values.} What is the order of $F_{\Phi,\varepsilon}(N)$ at $\varepsilon=4/5$ and at $\varepsilon=(\Phi^2-1)/(\Phi^2+1)$? At $\varepsilon=4/5$ we only know $\sqrt{N/8}/64<F_{\Phi,4/5}(N)\le N-1$, by monotonicity in $\varepsilon$. For $\Phi=5$ and $\varepsilon=12/13$ we only know that the dimension is at least of order $\log N$ and at most of order $\sqrt N\log N$.
\item \emph{General metrics.} Is the worst-case dimension $O(\sqrt N\,\mathrm{polylog}\,N)$ for all $N$-point metrics when $4/5<\varepsilon<12/13$? By Theorem~\ref{thm:upper} and Remark~\ref{rem:barrier} the remaining difficulty lies in metrics for which $\rho_k$ is large for all $k$ of order $\sqrt N$, that is, metrics with several scales of distances among the $\sqrt N$ nearest points of some point.
\item \emph{Intermediate range for larger spread.} For $\Phi>5$ and $12/13\le\varepsilon\le(\Phi^2-1)/(\Phi^2+1)$ the upper bound is of order $\sqrt N\log N$, while the lower bounds of Corollary~\ref{cor:lower} applicable to distances in $[1,\Phi]$ have smaller exponents; for example, of order $N^{1/3}$ for $\Phi\ge7$ and $\varepsilon<24/25$. In general the lower bounds for $\varepsilon<\varepsilon_\ell$ depend on the largest number of edges of a graph of girth greater than $\ell$.
\item \emph{The local ratio.} Is the factor $\rho_k^4$ in Theorem~\ref{thm:upper}, or the factor $\Phi^2$ in Corollary~\ref{cor:spread}, necessary?
\item \emph{The bound of Lemma~\ref{lem:deng}.} Can the factor $\varepsilon^{-2}\log N$ between the two sides of the inequality after Theorem~\ref{thm:Cfamily} be removed, in either direction?
\end{enumerate}

\paragraph{Related problems.} If one prescribes which squared intervals vanish, the relative error forces these to vanish exactly, and the dimension is bounded below, up to an additive constant, by the minimum rank of a symmetric matrix with this zero pattern. Theorem~\ref{thm:null} shows that already for the simplest zero pattern the dimension cannot be lowered at all. For the minimum rank with free diagonal, linear lower bounds for almost all graphs were proved in~\cite{HHMS2010} by counting zero patterns. Exact representations of graphs in $\R^{p,q}$ by two values of the squared interval are studied in~\cite{NSS2023}. The relation to sign-rank is described in Remark~\ref{rem:signrank}: for bipartite patterns $\pdim_\varepsilon+2$ is at least the sign-rank of the pattern, but for explicit matrices only the approximate rank gives a linear bound.

\end{document}

%% file: figures/fig-thresholds.tex
\begin{tikzpicture}[x=1cm,y=1cm]
  \node[anchor=west,font=\small\bfseries] at (0.0,4.45) {\figTup};
  \node[draw=figaccent,line width=0.8pt,fill=figaccent!8,rounded corners=2pt,align=center,minimum width=3.45cm,minimum height=0.85cm] at (2.15,3.35) {$\Theta(N)$};
  \node[draw=figaccent,line width=0.8pt,fill=black!4,rounded corners=2pt,align=center,minimum width=3.7cm,minimum height=0.85cm] at (6.15,3.35) {$\Omega(\sqrt N)$\\$O(\sqrt N\log N)$};
  \node[draw=figaccent,line width=0.8pt,fill=figaccent!8,rounded corners=2pt,align=center,minimum width=3.45cm,minimum height=0.85cm] at (10.1,3.35) {$\Theta(\log N)$};
  \draw[axis] (0.45,2.65) -- (11.95,2.65) node[right] {$\varepsilon$};
  \draw[thinline] (0.75,2.52)--(0.75,2.78);
  \draw[thinline] (4.15,2.52)--(4.15,2.78);
  \draw[thinline] (8.15,2.52)--(8.15,2.78);
  \draw[thinline] (11.5,2.52)--(11.5,2.78);
  \node[below] at (0.75,2.52) {$0$};
  \node[below] at (4.15,2.52) {$4/5$};
  \node[below] at (8.15,2.52) {$12/13$};
  \node[below] at (11.5,2.52) {$1$};
  \draw[thinline] (4.15,2.44)--(4.15,2.86);
  \draw[thinline] (8.15,2.44)--(8.15,2.86);
  \node[anchor=west,font=\small] at (0.0,1.62) {\figTnote};

  \node[anchor=west,font=\small\bfseries] at (0.0,0.92) {\figTlow};
  \node[draw=black!60,line width=0.7pt,fill=white,rounded corners=2pt,align=center,text width=11.4cm,minimum height=0.85cm] at (5.75,0.18) {\figTlowtext};
\end{tikzpicture}

%% file: figures/fig-landmarks.tex
\begin{tikzpicture}[x=1cm,y=1cm]
  \node[font=\small\bfseries] at (1.85,1.60) {\figLa};
  \draw[figaccent!18,fill=figaccent!6,rounded corners=3pt] (0.15,0.02) rectangle (1.68,0.60);
  \node[pt,label=below:$u$] (u) at (0.35,0.30) {};
  \node[pt] at (0.75,0.38) {};
  \node[pt] at (1.08,0.22) {};
  \node[pt] at (1.40,0.37) {};
  \node[font=\scriptsize] at (0.94,1.00) {$B(u)$};
  \node[draw,fill=figaccent,inner sep=3pt,label=below:$p(u)$] (p) at (2.25,0.30) {};
  \node[pt,label=below:$v$] (v) at (3.23,0.30) {};
  \draw[accent] (u) -- (p);
  \draw[aux] (p) -- (v);
  \node[font=\scriptsize,fill=white,inner sep=1pt] at (1.56,0.65) {$r_u$};

  \draw[aux] (3.83,-0.62) -- (3.83,1.15);

  \node[font=\small\bfseries] at (6.28,1.60) {\figLb};
  \node[draw,rounded corners=2pt,fill=white,line width=0.8pt,
        minimum width=1.10cm,minimum height=0.76cm,align=center] (T) at (4.72,0.35) {$\mat{T}$};
  \node[font=\Large] at (5.47,0.35) {$+$};
  \node[draw,rounded corners=2pt,fill=figaccent!7,line width=0.8pt,
        minimum width=1.10cm,minimum height=0.76cm,align=center] (E) at (6.28,0.35) {$\mat{E}$};
  \node[font=\Large] at (7.08,0.35) {$=$};
  \node[draw,rounded corners=2pt,fill=white,line width=0.8pt,
        minimum width=1.12cm,minimum height=0.76cm,align=center] (A) at (7.91,0.35) {$\mat{A}^{*}$};
  \node[font=\scriptsize,align=center,text width=1.8cm,anchor=north,execute at begin node=\strut] at (4.72,-0.14) {\figLlow};
  \node[font=\scriptsize,align=center,text width=1.8cm,anchor=north,execute at begin node=\strut] at (6.28,-0.14) {\figLsparse};
  \node[font=\scriptsize,align=center,text width=1.8cm,anchor=north,execute at begin node=\strut] at (7.91,-0.14) {\figLtarget};

  \draw[aux] (8.75,-0.62) -- (8.75,1.15);

  \node[font=\small\bfseries] at (11.15,1.60) {\figLc};
  \node[draw,rounded corners=2pt,fill=white,line width=0.8pt,
        align=center,minimum width=1.48cm,minimum height=0.76cm] (G) at (9.90,0.35)
        {$\widetilde{\mat{G}}$};
  \draw[accent,-{Stealth[length=5pt]}] (10.72,0.35) -- (11.12,0.35);
  \node[draw,rounded corners=2pt,fill=figaccent!7,line width=0.8pt,
        align=center,minimum width=2.05cm,minimum height=0.76cm] (K) at (12.20,0.35)
        {$\mathcal{K}(\widetilde{\mat{G}})$};
  \node[font=\scriptsize,align=center,text width=1.8cm,anchor=north,execute at begin node=\strut] at (9.90,-0.14) {\figLrankctl};
  \node[font=\scriptsize,align=center,text width=1.8cm,anchor=north,execute at begin node=\strut] at (12.20,-0.14) {\figLsq};
\end{tikzpicture}

%% file: figures/fig-tradeoff.tex
\begin{tikzpicture}
\begin{axis}[
  width=0.86\linewidth,
  height=5.2cm,
  xmin=0.35, xmax=2.35,
  ymin=0, ymax=4.4,
  axis lines=left,
  xlabel={$x=k/k_*$},
  ylabel={\figKy},
  xtick={0.5,1,1.5,2},
  ytick={1,2,3,4},
  tick label style={font=\small},
  label style={font=\small},
  grid=major,
  grid style={black!12},
  legend style={font=\small,draw=none,fill=white,at={(0.98,0.98)},anchor=north east},
  clip=true
]
  \addplot[accent,domain=0.35:2.35,samples=100,line width=1.2pt] {1/x};
  \addlegendentry{\figKa}
  \addplot[black,dashed,domain=0.35:2.35,samples=100,line width=1.0pt] {x};
  \addlegendentry{\figKb}
  \addplot[black!70,densely dotted,domain=0.35:2.35,samples=100,line width=1.5pt] {x+1/x};
  \addlegendentry{\figKsum}
  \addplot[aux] coordinates {(1,0) (1,3.3)};
\end{axis}
\end{tikzpicture}